\documentclass{article}
\usepackage[utf8]{inputenc}
\usepackage[utf8]{inputenc}
\usepackage[utf8]{inputenc}
\usepackage{dirtytalk}
\usepackage{tikz}
\usepackage{geometry}
\usepackage{graphicx}
\usepackage{enumitem}
\usepackage{parskip}
\usepackage{amsfonts}
\usepackage{amssymb}
\usepackage{amsmath}
\usepackage{amsthm}
\usepackage{xcolor}
\usepackage{hyperref}
\newlist{myitemize}{itemize}{1}
\setlist[myitemize,1]{leftmargin = 0.5in}
\hypersetup{colorlinks = true, linkcolor = red,  linktocpage}

\theoremstyle{plain}
\newtheorem{thm}{Theorem}[section]
\newtheorem*{thm*}{Theorem}

\newtheorem{cor}[thm]{Corollary}

\theoremstyle{definition}

\newtheorem{conj}[thm]{Conjecture}

\newtheorem{rem}[thm]{Remark}

\title{\textbf{\small{COUNTING THE NUMBER OF $n$-PERIODIC $\mathbb{Z}$-POINTS OF A DISCRETE DYNAMICAL SYSTEM WITH APPLICATIONS FROM ARITHMETIC STATISTICS, IV}}}
\author{\footnotesize{BRIAN KINTU}}
\date{\small{\textit{January 3, 2026}}}

\begin{document}
\maketitle
\begin{abstract}
\small{In this follow-up paper, we inspect a surprising relationship between the set of $n$-periodic points of a polynomial map $\varphi_{d, c}$ defined by $\varphi_{d, c}(z) = z^d + c$ for all $c, z \in \mathbb{Z}$ and the coefficient $c$, where $d>2$ is an integer and $n\geq 2$ is any fixed integer. As in \cite{BK3} we again wish to study counting problems which are inspired by the exciting advances on $n$-torsion point-counting in arithmetic statistics and on $n$-periodic point-counting in arithmetic dynamics. In doing so, we then first prove that for any prime $p\geq 3$ and for any fixed (period) $n\in \mathbb{Z}_{\geq 2}$, the average number of distinct $n$-periodic integral points of any $\varphi_{p, c}$ modulo $p$ is unbounded or zero as $c$ tends to infinity; and so the average behavior here coincide with the the average behavior of the number of distinct fixed integral points of any $\varphi_{p, c}$ modulo $p$ in \cite{BK3}. Inspired further by a conjecture of Hutz on $n$-periodic rational points of any $\varphi_{p-1, c}$ for any prime $p\geq 5$ in arithmetic dynamics, we then also prove that for any fixed (period) $n\in \mathbb{Z}_{\geq 2}$, the average number of distinct $n$-periodic integral points of any $\varphi_{p-1, c}$ modulo $p$ is $1$ or $2$ or $0$ as $c\to \infty$; and so the average behavior here also coincide with the average behavior of the number of distinct fixed integral points of any $\varphi_{p-1, c}$ modulo $p$ in \cite{BK1}. Finally, we apply density, polynomial-counting, number field-counting, and Sato-Tate equidistribution results from arithmetic statistics, and thereby obtaining a stream of counting and statistical results on the irreducible polynomials, number fields, and on Artin $L$-functions that arise naturally in our polynomial discrete dynamical settings.}
\end{abstract}

\begin{center}
\tableofcontents
\end{center}

\begin{center}
    \section{Introduction}\label{sec1}
\end{center}
\noindent
Given any morphism $\varphi: {\mathbb{P}^N(K)} \rightarrow {\mathbb{P}^N(K)} $ of degree $d \geq 2$ defined on a projective space ${\mathbb{P}^N(K)}$ of dimension $N$, where $K$ is a number field. Then for any $n\in\mathbb{Z}$ and $\alpha\in\mathbb{P}^N(K)$, we then call $\varphi^n = \underbrace{\varphi \circ \varphi \circ \cdots \circ \varphi}_\text{$n$ times}$ the $n^{th}$ \textit{iterate of $\varphi$} and call $\varphi^n(\alpha)$ the \textit{$n^{th}$ iteration of $\varphi$ on $\alpha$}. By convention, $\varphi^{0}$ acts as the identity map, i.e., $\varphi^{0}(\alpha) = \alpha$ for every point $\alpha\in {\mathbb{P}^N(K)}$. As before, the everyday philosopher may want to know (quoting here Devaney \cite{Dev}): \say{\textit{Where do points $\alpha, \varphi(\alpha), \varphi^2(\alpha), \ \cdots\ ,\varphi^n(\alpha)$ go as $n$ becomes large, and what do they do when they get there?}} Now for any given $n\in \mathbb{Z}_{\geq 0}$ and for any given point $\alpha\in {\mathbb{P}^N(K)}$, we then call the set consisting of all the iterates $\varphi^n(\alpha)$ the \textit{(forward) orbit of $\alpha$}; which in dynamical systems we usually denote it by $\mathcal{O}^{+}(\alpha)$.

As we mentioned in \cite{BK1} that one of the main 
goals in arithmetic dynamics is to classify all the points $\alpha\in\mathbb{P}^N(K)$ according to the behavior of their forward orbits $\mathcal{O}^{+}(\alpha)$. In this direction, we recall that any point $\alpha\in {\mathbb{P}^N(K)}$ is called a \textit{periodic point of $\varphi$}, whenever $\varphi^n (\alpha) = \alpha$ for some $n\in \mathbb{Z}_{\geq 0}$. In this case, we recall that any integer $n\geq 0$ such that $\varphi^n (\alpha) = \alpha$, is called \textit{period of $\alpha$}; and the smallest such integer $n\in \mathbb{Z}_{\geq 1}$ is called the \textit{exact period of $\alpha$}. We recall Per$(\varphi, {\mathbb{P}^N(K)})$ to denote set of all periodic points of $\varphi$; and also recall that for any given point $\alpha\in$Per$(\varphi, {\mathbb{P}^N(K)})$ the set of all iterates of $\varphi$ on $\alpha$ is called \textit{periodic orbit of $\alpha$}. In their 1994 paper \cite{Russo} and in his 1998 paper \cite{Poonen} resp., Walde-Russo and Poonen give independently interesting examples of rational periodic points of any $\varphi_{2,c}$ defined over $\mathbb{Q}$; and so the interested reader may wish to revisit \cite{Russo, Poonen}.

Previously in article \cite{BK3} we (inspired by work of Bhargava-Shankar-Tsimerman (BST) in arithmetic statistics, and also by Narkiewicz' argument of Theorem \ref{theorem 3.2.1} in arithmetic dynamics) proved that the number of distinct fixed integral points of any $\varphi_{p,c}$ modulo $p$ is equal to $p$ or $0$; from which it then followed that the average number of distinct fixed integral points of any $\varphi_{p,c}$ modulo $p$ is unbounded or zero as $c\to \infty$. Moreover, we then also observed in \cite{BK3} that the expected total number of distinct fixed integral points in the whole family of polynomial maps $\varphi_{p,c}$ modulo $p$ is equal to $p+0=p$; which may grow to infinity whenever deg$(\varphi_{p,c})=p\to \infty$. So now, inspired by celebrated work of Mazur \cite{Maz} on $n$-torsion points of elliptic curves, and also again inspired by the advances of (BST) on $n$-torsion of arithmetic objects in arithmetic statistics and also by \ref{per} of Morton-Silverman's Conjecture \ref{silver-morton} in arithmetic dynamics, we (as in \cite{BK1}) wish to again study in Section \ref{sec4}, \ref{sec5}, \ref{sec6}, \ref{sec8}, \ref{sec9}, \ref{sec10}, \ref{sec11} and \ref{sec12} somewhat analogous counting and statistical questions; which among them includes again the natural question: \say{\textit{For any fixed (period) $n\geq 2$, how many distinct $n$-periodic integral orbits can any $\varphi_{p,c}$ and $\varphi_{p-1,c}$ acting independently on the space $\mathbb{Z} / p\mathbb{Z}$ via iteration have on average as $c \to \infty$?}} In doing so, we then first prove the following main theorem on every $\varphi_{p,c}$; and which we state later more precisely as Theorem \ref{2.2}:

\begin{thm}\label{Binder-Brian1}
Let $p\geq 3$ be any fixed prime, and $n\geq 2$ be any fixed integer. Let $\varphi_{p, c}$ be defined by $\varphi_{p, c}(z) = z^p+c$ for all $c, z\in\mathbb{Z}$. Then the number of distinct $n$-periodic integral points of any $\varphi_{p,c}$ modulo $p$ is $p$ or zero.
\end{thm}

Recall in article \cite{BK1} we (inspired by (BST)'s work in arithmetic statistics, and also by Hutz's Conjecture \ref{conjecture 3.2.1} along with Panraksa's work \cite{par2} in arithmetic dynamics) 
proved that the number of distinct fixed integral points of any $\varphi_{p-1,c}$ modulo $p$ is equal to $1$ or $2$ or $0$; from which it then followed that the average number of distinct fixed integral points of any $\varphi_{p-1,c}$ modulo $p$ is also equal to $1$ or $2$ or $0$ as $p\to \infty$. Moreover, we then also observed in [\cite{BK1}, Remark 3.3] that the expected total number of distinct fixed integral points in the whole family of maps $\varphi_{p-1,c}$ modulo $p$ is equal to $1 + 2+ 0 =3$; and so the expected total count is independent of $p$ and hence of deg$(\varphi_{p-1}, c)=p-1$. So now, motivated again by Mazur's work \cite{Maz} and (BST)'s work on $n$-torsion point-counting in arithmetic statistics, along with Hutz's Conjecture \ref{conjecture 3.2.1} and Panraksa's work \cite{par2} in arithmetic dynamics, we revisit the setting in Sect.\ref{sec2} and consider in Section \ref{sec3} any $\varphi_{p-1,c}$. In doing so, we then also prove the following main theorem on any $\varphi_{p-1,c}$; and which we state later more precisely as Theorem \ref{3.2}:

\begin{thm}\label{Binder-Brian2}
Let $p\geq 5$ be any fixed prime, and $n\geq 2$ be any fixed integer. Let $\varphi_{p-1, c}$ be defined by $\varphi_{p-1, c}(z)$ for all $c, z\in\mathbb{Z}$. Then the number of distinct $n$-periodic integral points of any $\varphi_{p-1,c}$ modulo $p$ is $1$ or $2$ or zero.
\end{thm}

\noindent Notice that the count obtained in Theorem \ref{Binder-Brian2} and more precisely in Theorem \ref{3.2} on the number of distinct $n$-periodic integral points of any $\varphi_{p-1,c}$ modulo $p$ is independent of $p$ (and so independent of deg$(\varphi_{p-1,c}))$ in each of the possibilities. Moreover, we may also observe that the expected total count (namely, $1 + 2 + 0 =3$ for every fixed odd period $n\in \mathbb{Z}_{\geq 3}$ or $1 + 1 + 2 + 0 =4$ for every fixed even period $n\in \mathbb{Z}_{\geq 2}$) in Theorem \ref{3.2} (and so in Theorem \ref{Binder-Brian2}) on the number of distinct $n$-periodic integral points in the whole family of maps $\varphi_{p-1,c}$ modulo $p$ is also independent of $p$ and deg$(\varphi_{p-1,c})$. On the other hand, we may also notice that the count obtained in Theorem \ref{Binder-Brian1} on the number of distinct $n$-periodic integral points of any $\varphi_{p,c}$ modulo $p$ may depend on $p$ (and thus depend on deg$(\varphi_{p,c}))$ in one of the possibilities. As a result, the expected total count (namely, $p+0 =p$ for every fixed period $n\in \mathbb{Z}_{\geq 2}$) in Theorem \ref{Binder-Brian1} on the number of distinct $n$-periodic integral points in the whole family of maps $\varphi_{p,c}$ modulo $p$ may not only depend on the degree $p$, but also may grow to infinity as $p\to \infty$.

Motivated by work of Adam-Fares \cite{Ada} in arithmetic dynamics and also by a \say{counting-application} philosophy in arithmetic statistics, we then inspect in a forthcoming article \cite{BK333} the aforementioned relationship where we consider the space $\mathbb{Z}_{p}$ of all $p$-adic integers. In doing so, we then prove $n$-periodic integral point-counting result and asymptotics on any $\varphi_{p,c}$ iterated on $\mathbb{Z}_{p}\slash p\mathbb{Z}_{p}$ that's very analogous to the counting and asymptotics proved in this article, and also prove $n$-periodic integral point-counting result and asymptotics on any $\varphi_{p-1,c}$ iterated on $\mathbb{Z}_{p}\slash p\mathbb{Z}_{p}$ that's also very analogous to the counting and asymptotics proved in this same article. Inspired by work of Narkiewicz \cite{Narkie1} along with $K$-rational periodic version of Morton-Silverman's Conjecture \ref{silver-morton}, we then in upcoming work \cite{BK222} revisit the setting in Section \ref{sec2} and then consider any $\varphi_{p^{\ell},c}$ defined over any number field $K\slash \mathbb{Q}$ of degree $n\geq 2$, where $p\geq 3$ is any prime and $\ell \geq 1$ is any integer. In doing so, we then also prove $m$-periodic integral point count that's not only very analogous to the count in Theorem \ref{Binder-Brian1}, but also may grow to infinity as $p\to \infty$. More in \cite{BK222}, we again revisit the setting in Section \ref{sec3} and then consider any $\varphi_{(p-1)^{\ell},c}$ defined over any number field $K\slash \mathbb{Q}$ of degree $n\geq 2$, where $p\geq 5$ is any prime and $\ell \geq 1$ is any integer. In doing so, we then also prove $m$-periodic integral point count that is not only very analogous to the count in Theorem \ref{Binder-Brian2} for every $K$, every $m\geq 2$, every prime $p$ and every $\ell\in \mathbb{Z}_{\geq 1}$, but is also independent of $n, m$ and $p$.

In addition, to the notion of a periodic point and a periodic orbit, we also recall that a point $\alpha\in {\mathbb{P}^N(K)}$ is called a \textit{preperiodic point of $\varphi$}, whenever $\varphi^{m+n}(\alpha) = \varphi^{m}(\alpha)$ for some integers $m\geq 0$ and $n\geq 1$. In this case, we recall that the smallest integers $m\geq 0$ and $n\geq 1$ such that $\varphi^{m+n}(\alpha) = \varphi^{m}(\alpha)$, are called the \textit{preperiod} and \textit{eventual period of $\alpha$}, resp. Again, we denote the set of preperiodic points of $\varphi$ by PrePer$(\varphi, {\mathbb{P}^N(K)})$. For any given preperiodic point $\alpha$ of $\varphi$, we then call the set of all iterates of $\varphi$ on $\alpha$, \textit{the preperiodic orbit of $\alpha$}.
Now observe for $m=0$, we have $\varphi^{n}(\alpha) = \alpha $ and so $\alpha$ is a periodic point of period $n$. Thus, the set  Per$(\varphi, {\mathbb{P}^N(K)}) \subseteq$ PrePer$(\varphi, {\mathbb{P}^N(K)})$; however, it need not be PrePer$(\varphi, {\mathbb{P}^N(K)})\subseteq$ Per$(\varphi, {\mathbb{P}^N(K)})$. In their 2014 paper \cite{Doyle}, Doyle-Faber-Krumm give nice examples (which also recovers examples in Poonen's paper \cite{Poonen}) of preperiodic points of any quadratic map $\varphi$ defined over quadratic fields; and which the interested reader may also wish to revisit. 

In the year 1950, Northcott \cite{North} used the theory of height functions to show that not only is the set PrePer$(\varphi, {\mathbb{P}^N(K)})$ always finite, but also for a given morphism $\varphi$ the set PrePer$(\varphi, {\mathbb{P}^N(K)})$ can be computed effectively. Forty-five years later, in the year 1995, Morton and Silverman conjectured that PrePer$(\varphi, \mathbb{P}^N(K))$ can be bounded in terms of degree $d$ of $\varphi$, degree $D$ of $K$, and dimension $N$ of the space ${\mathbb{P}^N(K)}$. This celebrated conjecture is called the \textit{Uniform Boundedness Conjecture}; which we then restate here as the following conjecture:

\begin{conj} \label{silver-morton}[\cite{Morton}]
Fix integers $D \geq 1$, $N \geq 1$, and $d \geq 2$. There exists a constant $C'= C'(D, N, d)$ such that for all number fields $K/{\mathbb{Q}}$ of degree at most $D$, and all morphisms $\varphi: {\mathbb{P}^N}(K) \rightarrow {\mathbb{P}^N}(K)$ of degree $d$ defined over $K$, the total number of preperiodic points of a morphism $\varphi$ is at most $C'$, i.e., \#PrePer$(\varphi, \mathbb{P}^N(K)) \leq C'$.
\end{conj}
\noindent Note that a special case of Conjecture \ref{silver-morton} is when the degree $D$ of a number field $K$ is $D = 1$, dimension $N$ of a space $\mathbb{P}^N(K)$ is $N = 1$, and degree $d$ of a morphism $\varphi$ is $d = 2$. In this case, if $\varphi$ is a polynomial morphism, then it is a quadratic map defined over the field $\mathbb{Q}$. Moreover, in this very special case, in the year 1995, Flynn and Poonen and Schaefer conjectured that a quadratic map has no points $z\in\mathbb{Q}$ with exact period more than 3. This conjecture of Flynn-Poonen-Schaefer \cite{Flynn} (which has been resolved for cases $n = 4$, $5$ in \cite{mor, Flynn} respectively and conditionally for $n=6$ in \cite{Stoll} is, however, still open for all cases $n\geq 7$ and moreover, which also Hutz-Ingram \cite{Ingram} gave strong computational evidence supporting it) is restated here formally as the following conjecture. Note that in this same special case, rational points of exact period $n\in \{1, 2, 3\}$ were first found in the year 1994 by Russo-Walde \cite{Russo} and also found in the year 1995 by Poonen \cite{Poonen} using a different set of techniques. We restate here the anticipated conjecture of Flynn-Poonen-Schaefer as the following conjecture: 
 
\begin{conj} \label{conj:2.4.1}[\cite{Flynn}, Conjecture 2]
If $n \geq 4$, then there is no quadratic polynomial $\varphi_{2,c }(z) = z^2 + c\in \mathbb{Q}[z]$ with a rational point of exact period $n$.
\end{conj}
Now by assuming Conjecture \ref{conj:2.4.1} and also establishing interesting results on rational preperiodic points, in the year 1998, Poonen \cite{Poonen} then concluded that the total number of rational preperiodic points of any quadratic polynomial $\varphi_{2, c}(z)=z^2 + c$ is at most nine. We restate here formally Poonen's result as the following corollary:
\begin{cor}\label{cor2}[\cite{Poonen}, Corollary 1]
If Conjecture \ref{conj:2.4.1} holds, then $\#$PrePer$(\varphi_{2,c}, \mathbb{Q}) \leq 9$,  for all quadratic maps $\varphi_{2, c}$ defined by $\varphi_{2, c}(z) = z^2 + c$ for all points $c, z\in\mathbb{Q}$.
\end{cor}

Recall that Per$(\varphi, {\mathbb{P}^N(K)}) \subseteq$ PrePer$(\varphi, {\mathbb{P}^N(K)})$ and if $\#$PrePer$(\varphi, \mathbb{P}^N(K))$ is bounded above, then so is $\#$Per$(\varphi, \mathbb{P}^N(K))$ also bounded above by the same upper bound. So now, recall in \cite{BK1, BK3} we attempted to understand (on the level of $\mathbb{Z}$) the possibility and validity of Conjecture \ref{per}. In this article, we again wish to continue with this adventure of hoping to understand (again on the level of $\mathbb{Z}$) the possibility and validity of  Conjecture \ref{per}. That is, in Section \ref{sec2} (resp. \ref{sec3}) we study a dynamical setting in which $K$ is replaced with $\mathbb{Z}$, $N=1$ and $d=p$ (resp. $d = p-1$); and again hoping to understand the possibility and validity of the following:

\begin{conj} \label{silver-morton 1}($(D, N) = (1,1)$-version of Conjecture \ref{silver-morton})\label{per}
Fix an integer $d \geq 2$. There exists a constant $C'= C'(d)$ such that for all morphisms $\varphi: {\mathbb{P}}^1(\mathbb{Q}) \rightarrow {\mathbb{P}}^1(\mathbb{Q})$ of degree $d$, the number $\#$Per$(\varphi, {\mathbb{P}}^1(\mathbb{Q})) \leq C'(d)$.
\end{conj}

\subsection*{History on the Connection Between the Size of Per$(\varphi_{d, c}, K)$ and the Coefficient $c$}

In the year 1994, Walde and Russo not only proved [\cite{Russo}, Corollary 4] that for a quadratic map $\varphi_{2,c}$ defined over $\mathbb{Q}$ with a periodic point, the denominator of a rational point $c$, denoted as den$(c)$, is a square but they also proved that den$(c)$ is even, whenever $\varphi_{2,c}$ admits a rational cycle of length $\ell \geq 3$. Moreover, Walde-Russo also proved [\cite{Russo}, Cor. 6, Thm 8 and Cor. 7] that the size \#Per$(\varphi_{2, c}, \mathbb{Q})\leq 2$, whenever den$(c)$ is an odd integer. 

Three years later, in the year 1997, Call-Goldstine \cite{Call} proved that the size of PrePer$(\varphi_{2,c},\mathbb{Q})$ can be bounded above in terms of the number of distinct odd primes dividing den$(c)$. We restate formally this result of Call-Goldstine as the following theorem, in which $GCD(a, e)$ refers to the greatest common divisor of $a$, $e \in \mathbb{Z}$:

\begin{thm}\label{2.3.1}[\cite{Call}, Theorem 6.9]
Let $e>0$ be an integer and let $s$ be the number of distinct odd prime factors of e. Define $\varepsilon  = 0$, $1$, $2$, if $4\nmid e$, if $4\mid e$ and $8 \nmid e$, if $8 \mid e$, respectively. Let $c = a/e^2$, where $a\in \mathbb{Z}$ and $GCD(a, e) = 1$. If $c \neq -2$, then the total number of $\mathbb{Q}$-preperiodic points of $\varphi_{2, c}$ is at most $2^{s + 2 + \varepsilon} + 1$. Moreover, a quadratic map $\varphi_{2, -2}$ has exactly six rational preperiodic points.
\end{thm}

Eight years later, after the work of Call-Goldstine, in the year 2005, Benedetto \cite{detto} studied polynomial maps $\varphi$ of arbitrary degree $d\geq 2$ defined over an arbitrary global field $K$, and then established the following result on the relationship between the size of the set PrePre$(\varphi, K)$ and the number of bad primes of $\varphi$ in $K$:

\begin{thm}\label{main} [\cite{detto}, Main Theorem]
Let $K$ be a global field, $\varphi\in K[z]$ be a polynomial of degree $d\geq 2$ and $s$ be the number of bad primes of $\varphi$ in $K$. The number of preperiodic points of $\varphi$ in $\mathbb{P}^N(K)$ is at most $O(\text{s log s})$. 
\end{thm}

\noindent Since Benedetto's Theorem \ref{main} applies to any polynomial $\varphi$ of arbitrary degree $d\geq 2$ defined over any number field $K$, it also then follows that one can immediately apply Benedetto's Theorem \ref{main} to any polynomial $\varphi$ of arbitrary odd or even degree $d> 2$ defined over any $K$ and as a result obtain the upper bound in Theorem \ref{main}.

Seven years after the work of Benedetto, in the year 2012, Narkiewicz's work \cite{Narkie} not only showed that any $\varphi_{d,c}$ defined over $\mathbb{Q}$ with odd degree $d\geq 3$ has no rational periodic points of exact period $n > 1$, but his also showed that the total number of $\mathbb{Q}$-preperiodic points is at most 4. We restate this result here as the following: 

\begin{thm} \label{theorem 3.2.1}\cite{Narkie}
For any integer $n > 1$ and any odd integer $d\geq 3$, there is no $c\in \mathbb{Q}$ such that $\varphi_{d,c}$ defined by $\varphi_{d, c}(z)$ for all $c,z \in \mathbb{Q}$ has rational periodic points of exact period $n$. Moreover, $\#PrePer(\varphi_{d, c}, \mathbb{Q}) \leq 4$. 
\end{thm} 

Three years after \cite{Narkie}, in 2015, Hutz \cite{Hutz} developed an algorithm determining effectively all $\mathbb{Q}$-preperiodic points of a morphism defined over a given number field $K$; from which he then made the following conjecture: 

\begin{conj} \label{conjecture 3.2.1}[\cite{Hutz}, Conjecture 1a]
For any integer $n > 2$, there is no even degree $d > 2$ and no point $c \in \mathbb{Q}$ such that the polynomial map $\varphi_{d, c}$ has rational points of exact period $n$.
Moreover, \#PrePer$(\varphi_{d, c}, \mathbb{Q}) \leq 4$. 
\end{conj}

\begin{rem}\label{ht}
If Conjecture \ref{conjecture 3.2.1} held, it would then follow that the total number of $2$-periodic rational points (and so the total number of $2$-periodic integral points) of any $\varphi_{d,c}(x) = x^d + c$ of even degree $d>2$ is equal to 4 or strictly less than 4. Moreover, since the monic polynomial $\varphi_{d,c}(x)\in \mathbb{Z}[x]$ has good reduction modulo $p$, then the total number of $2$-periodic integral points of any $\varphi_{d,c}(x)$ modulo $p$ (and hence of any $\varphi_{d,c}$ modulo $p$) is also equal to 4 or strictly less than 4. Furthermore, if Conjecture \ref{conjecture 3.2.1} held, then it would also follow that the total number of fixed integral points and $2$-periodic integral points of any $\varphi_{d,c}(x)$ modulo $p$ (and hence of any $\varphi_{d,c}$ modulo $p$) is also equal to 4 or strictly less than 4. But of course now the issue is that we unfortunately don't know as of this article whether Conjecture \ref{conjecture 3.2.1} holds or not, let alone whether $4$ is the correct upper bound on the total number of fixed and $2$-periodic rational (and hence integral) points of any $\varphi_{d,c}$ of even degree $d>2$. Additionally, if Conjecture \ref{conjecture 3.2.1} held, then this would also mean that for any integer $n\geq 3$ and for any even integer $d\geq 4$, then $\varphi_{d, c}^n(z)-z \neq 0$ for all rational points $c, z\in \mathbb{Q}$ (and hence for all integral points $z\in \mathbb{Z}$). This would then also mean that the total number of $n$-periodic rational points (and hence the total number of $n$-periodic integral points) of any $\varphi_{d, c}$ of even degree $d\geq 4$ is equal to zero. Moreover, since the monic polynomial $\varphi_{d,c}(x)\in \mathbb{Z}[x]$ has good reduction modulo any fixed prime $p$ and so the reduced polynomial $\varphi_{d,c}(x)$ modulo $p$ also has even degree $d\geq 4$, it would then also follow that the total number of $n$-periodic integral points of any $\varphi_{d,c}(x)$ modulo $p$ (and hence of any $\varphi_{d,c}$ modulo $p$) is equal to zero. But now, as before the issue here is that we unfortunately don't know (as to the author's knowledge) whether Conjecture \ref{conjecture 3.2.1} holds or not. On the note whether any theoretical progress has yet been made on Conjecture \ref{conjecture 3.2.1}, more recently, Panraksa \cite{par2} proved among many other results that the quartic polynomial $\varphi_{4,c}(z)\in\mathbb{Q}[z]$ has rational points of exact period $n = 2$. Moreover, he also proved that $\varphi_{d,c}(z)\in\mathbb{Q}[z]$ has no rational points of exact period $n = 2$ for every $c \in \mathbb{Q}$ with $c \neq -1$ and $d = 6$, $2k$ with $3 \mid 2k-1$. The interested reader may find these mentioned results of Panraksa in his unconditional Theorems 2.1, 2.4 and Theorem 1.7 conditioned on the abc-conjecture in \cite{par2}.
\end{rem}

Twenty-eight years later, after the work of Walde-Russo, in the year 2022, Eliahou-Fares proved [\cite{Shalom2}, Theorem 2.12] that the denominator of a rational point $-c$, denoted as den$(-c)$ is divisible by 16, whenever $\varphi_{2,-c}$ defined by $\varphi_{2, -c}(z) = z^2 - c$ for all $c, z\in \mathbb{Q}$ admits a rational cycle of length $\ell \geq 3$. Moreover, they also proved [\cite{Shalom2}, Proposition 2.8] that the size \#Per$(\varphi_{2, -c}, \mathbb{Q})\leq 2$, whenever den$(-c)$ is an odd integer. Motivated by \cite{Call}, Eliahou-Fares \cite{Shalom2} also proved that the size of Per$(\varphi_{2, -c}, \mathbb{Q})$ can be bounded above by using information on den$(-c)$, namely, information in terms of the number of distinct primes dividing den$(-c)$. Moreover, they in \cite{Shalom1} also showed that the upper bound is four, whenever $c\in \mathbb{Q^*} = \mathbb{Q}\setminus\{0\}$. We restate here their results as:

\begin{cor}\label{sha}[\cite{Shalom2, Shalom1}, Cor. 3.11 and Cor. 4.4, respectively]
Let $c\in \mathbb{Q}$ such that den$(c) = d^2$ with $d\in 4 \mathbb{N}$. Let $s$ be the number of distinct primes dividing $d$. Then, the total number of $\mathbb{Q}$-periodic points of $\varphi_{2, -c}$ is at most $2^s + 2$. Moreover, for $c\in \mathbb{Q^*}$ such that the den$(c)$ is a power of a prime number. Then, $\#$Per$(\varphi_{2, c}, \mathbb{Q}) \leq 4$.
\end{cor}

\noindent Once again, the purpose of this article is (as in \cite{BK1, BK3}) to inspect further the above connection, independently in the case of polynomial maps $\varphi_{p, c}$ of odd prime degree $p$ defined over $\mathbb{Z}$ for any given prime integer $p\geq 3$ and in the case of polynomial maps $\varphi_{p-1, c}$ of even degree $p-1$ defined over $\mathbb{Z}$ for any given prime integer $p\geq 5$; and doing so from a spirit that is truly inspired by some of the many striking developments in arithmetic statistics.

\section{The Number of $n$-Periodic Integral Points of any Family of Polynomial Maps $\varphi_{p,c}$}\label{sec2}

In this section, we wish to count the number of distinct $n$-periodic integral points of any $\varphi_{p,c}$ modulo $p$, where $p\geq 3$ is any given prime and $n\geq 2$ is any fixed integer. With that in mind, we let $p\geq 3$ be any prime, $c\in \mathbb{Z}$ be any integer and $n\geq 2$ be any fixed integer, and then define the following $n$-periodic point-counting function 
\begin{equation}\label{N_{c}}
N_{c}^{(n)}(p) := \# \Biggl\{ z\in \mathbb{Z}\slash p\mathbb{Z}  : \begin{aligned} \varphi_{p,c}^{n-1}(z) -z \not \equiv 0 \ \text{(mod $p$)} \\ \ \varphi_{p,c}^{n}(z) - z \equiv 0 \ \text{(mod $p$)} \end{aligned} \Biggr\}.
\end{equation}\noindent We then first prove the following theorem on the number of distinct $n$-periodic points of any $\varphi_{3, c}$ modulo $3$:

\begin{thm} \label{2.1}
Let $\varphi_{3, c}$ be a cubic map defined by $\varphi_{3, c}(z) = z^3 + c$ for all $c, z\in\mathbb{Z}$, and let $N_{c}^{(n)}(3)$ be the number defined as in \textnormal{(\ref{N_{c}})}. Then $N_{c}^{(n)}(3) = 3$ for every coefficient $c = 3t$; otherwise $N_{c}^{(n)}(3) = 0$ for every $c \neq  3t$. 
\end{thm}
\begin{proof}
Let $f(z)= \varphi_{3,c}^n(z)-z = \varphi_{3,c}(\varphi_{3,c}^{n-1}(z)) - z = (\varphi_{3,c}^{n-1}(z))^3 - z + c$, and so $f(z)= (\varphi_{3,c}^{n-1}(z))^3 - z + c$. Now applying the multinomial theorem repeatedly on the term $(\varphi_{3,c}^{n-1}(z))^3$ right after applying the binomial theorem on $(z^3 + c)^3$, we then obtain that $(\varphi_{3,c}^{n-1}(z))^3$ is a monic polynomial in $z$ of degree $3^n$ with integer coefficients in multiples of $c$. Thus, we may then write $(\varphi_{3,c}^{n-1}(z))^3 = z^{3^{n}} + h(z)$, where $h(z)$ is a non-constant polynomial in $z$ of degree deg$(h)<3^n$ with integer coefficients in multiples of $c$; and so we then obtain $f(z)= z^{3^{n}} + h(z) - z + c$. Now for every coefficient $c=3t$, reducing $f(z)$ modulo $3$, it then follows $f(z)\equiv z^{3^n} - z$ (mod $3$), since also $h(z)\in c\mathbb{Z}[z]$ and so $h(z)\equiv 0$ (mod $3$); and so $f(z)$ modulo $3$ is now a polynomial defined over a finite field $\mathbb{Z}\slash 3\mathbb{Z}$ of order $3$. So now, recall by Fermat's Little Theorem (FLT) that $z^3 \equiv z$ (mod $3$) for every $z\in \mathbb{Z}$ (equivalently, $z^3 = z$ for every elment $z\in \mathbb{Z}\slash3\mathbb{Z}$), it then follows $z^{3^n}= (z^3)^{3^{n-1}} = z^{3^{n-1}} = (z^3)^{3^{n-2}} = z^{3^{n-2}}$ for every element $z\in \mathbb{Z}\slash3\mathbb{Z}$. Now since $n\geq 2$ and so $n-2\geq 0$, then if $n-2 = 0$ and so $z^{3^{n-2}} = z$, then this yields $z^{3^n} = z$  for every element $z\in \mathbb{Z}\slash3\mathbb{Z}$. But now the reduced polynomial $f(z)\equiv 0$ (mod $3$) for every point $z\in \mathbb{Z}\slash3\mathbb{Z}$. Otherwise, if $n-2 > 0$, then since $n$ is a fixed integer, we may then continue performing the above procedure of decreasing the exponent $n-2$ of $z^{3^{n-2}}= z^{3^n}$ for every element $z\in \mathbb{Z}\slash3\mathbb{Z}$, until $n-2$ is equal to zero; from which we then again obtain that $f(z)\equiv 0$ (mod $3$) for every $z\in \mathbb{Z}\slash3\mathbb{Z}$. But now, we then conclude $N_{c}^{(n)}(3) = 3$. 

Finally, we now show $N_{c}^{(n)}(3) = 0$ for every coefficient $c \neq  3t$ and for every fixed integer $n\geq 2$. For the sake of a contradiction, let's suppose $f(z)\equiv 0$ (mod $3$) for some point $z\in \mathbb{Z}\slash 3\mathbb{Z}$ and for every coefficient $c\not \equiv 0$ (mod 3) and every fixed $n\in \mathbb{Z}_{\geq 2}$. Then this also means $z^{3^{n}} + h(z) - z + c \equiv 0$ (mod $3$) for some $z\in \mathbb{Z}\slash 3\mathbb{Z}$ and for every $c\not \equiv 0$ (mod 3) and every fixed $n\geq 2$. So now, recall from earlier  that $z^{3^n} = z$ for every $z\in \mathbb{Z}\slash3\mathbb{Z}$ and fixed $n$, we may then rewrite $(z^{3^{n}}-z) + (h(z) + c) \equiv 0$ (mod 3) for some $z\in \mathbb{Z}\slash 3\mathbb{Z}$ and for every $c\not \equiv 0$ (mod 3) to then obtain the congruence $h(z) + c \equiv 0$ (mod 3) for some $z\in \mathbb{Z}\slash 3\mathbb{Z}$ and every $c\not \equiv 0$ (mod 3). Now looking at the multinomial expansion of $(\varphi_{3,c}^{n-1}(z))^3$, we then obtain $h(z)\equiv \sum_{i=1}^{n-1}c^{3^{n-i}}$ (mod $3$); and so $h(z) + c \equiv \sum_{i=1}^{n-1}c^{3^{n-i}} + c$ (mod $3$). Now since also $h(z) + c \equiv 0$ (mod 3), it then also follows $\sum_{i=1}^{n-1}c^{3^{n-i}} + c \equiv 0$ (mod $3$). But now observe $\sum_{i=1}^{n-1}c^{3^{n-i}} + c \equiv 0$ (mod $3$) can happen if $\sum_{i=1}^{n-1}c^{3^{n-i}}\equiv 0$ (mod $3$) and also $c\equiv 0$ (mod $3$); and so a contradiction. It then follows $f(x)=\varphi_{3,c}^n(x)-x$ has no roots in $\mathbb{Z}\slash 3\mathbb{Z}$ for every $c\neq 3t$ and for every fixed $n\in \mathbb{Z}_{\geq 2}$; and so we conclude $N_{c}^{(n)}(3) = 0$. This then completes the whole proof, as needed.
\end{proof} 
We now wish to generalize Theorem \ref{2.1} to any polynomial map $\varphi_{p, c}$ for any given prime $p\geq 3$. More precisely, we prove that the number of distinct $n$-periodic integral points of any $\varphi_{p, c}$ modulo $p$ is either $p$ or $0$:

\begin{thm} \label{2.2}
Let $p\geq 3$ be any fixed prime, and let $\varphi_{p, c}$ be defined by $\varphi_{p, c}(z) = z^p+c$ for all $c, z\in\mathbb{Z}$. Let $N_{c}^{(n)}(p)$ be as in \textnormal{(\ref{N_{c}})}. Then $N_{c}^{(n)}(p)=p$ for every coefficient $c = pt$; otherwise $N_{c}^{(n)}(p) = 0$ for all points $c \neq  pt$.
\end{thm}
\begin{proof}
By applying a similar argument as in the Proof of Theorem \ref{2.1}, we then obtain the count as desired. That is, let $f(z)= \varphi_{p,c}^n(z)-z = \varphi_{p,c}(\varphi_{p,c}^{n-1}(z)) - z = (\varphi_{p,c}^{n-1}(z))^p - z + c$, and so $f(z)= (\varphi_{p,c}^{n-1}(z))^p - z + c$. So now, applying the multinomial theorem repeatedly on $(\varphi_{p,c}^{n-1}(z))^p$ right after applying the binomial theorem on $(z^p + c)^p$, it then follows that $(\varphi_{p,c}^{n-1}(z))^p$ is a monic polynomial in $z$ of degree $p^n$ with integer coefficients in multiples of $c$. Thus, we may then write $(\varphi_{p,c}^{n-1}(z))^p = z^{p^{n}} + h(z)$, where $h(z)$ is a non-constant polynomial in $z$ of deg$(h)<p^n$ with integer coefficients in multiples of $c$; and so we then obtain $f(z)= z^{p^{n}} + h(z) - z + c$. So now, for every coefficient $c=pt$, reducing $f(z)$ modulo $p$, it then follows $f(z)\equiv z^{p^n} - z$ (mod $p$), since also $h(z)\in c\mathbb{Z}[z]$ and so $h(z)\equiv 0$ (mod $p$); and so $f(z)$ modulo $p$ is now a polynomial defined over a finite field $\mathbb{Z}\slash p\mathbb{Z}$ of order $p$. Now since $z^p = z$ for every element $z\in \mathbb{Z}\slash p\mathbb{Z}$, it then follows $z^{p^n}= (z^p)^{p^{n-1}} = z^{p^{n-1}}=(z^p)^{p^{n-2}} =z^{p^{n-2}}$ for every element $z\in \mathbb{Z}\slash p\mathbb{Z}$. So now, since $n\geq 2$ and so $n-2\geq 0$, then if $n-2 = 0$ and so $z^{p^{n-2}} = z$, then this yields $z^{p^n} = z$  for every element $z\in \mathbb{Z}\slash p\mathbb{Z}$. But then $f(z)\equiv 0$ (mod $p$) for every point $z\in \mathbb{Z}\slash p\mathbb{Z}$. Otherwise, if $n-2 > 0$, then since $n$ is a fixed integer, we may then continue performing the above procedure of decreasing the exponent $n-2$ of $z^{p^{n-2}}= z^{p^n}$ for every element $z\in \mathbb{Z}\slash p\mathbb{Z}$, until $n-2$ is equal to zero; and from which we then again obtain that $f(z)\equiv 0$ (mod $p$) for every point $z\in \mathbb{Z}\slash p\mathbb{Z}$. But now, we then conclude $N_{c}^{(n)}(p) = p$. 

Finally, we now show $N_{c}^{(n)}(p) = 0$ for every coefficient $c \neq  pt$ and every fixed $n\geq 2$. As before, let's for the sake of a contradiction, suppose $f(z)\equiv 0$ (mod $p$) for some $z\in \mathbb{Z}\slash p\mathbb{Z}$ and for every $c\not \equiv 0$ (mod $p$) and fixed $n\geq 2$. Then $z^{p^{n}} + h(z) - z + c \equiv 0$ (mod $p$) for some $z\in \mathbb{Z}\slash p\mathbb{Z}$ and for every $c\not \equiv 0$ (mod $p$) and fixed $n\geq 2$. Now since $z^{p^n} = z$ for every $z\in \mathbb{Z}\slash p\mathbb{Z}$ and fixed $n\geq 2$, we may then rewrite $(z^{p^{n}} -z) + (h(z) + c) \equiv 0$ (mod $p$) for some $z\in \mathbb{Z}\slash p\mathbb{Z}$ and every $c\not \equiv 0$ (mod $p$) to then obtain $h(z) + c \equiv 0$ (mod $p$) for some $z\in \mathbb{Z}\slash p\mathbb{Z}$ and every $c\not \equiv 0$ (mod $p$). Now looking at the multinomial expansion of $(\varphi_{p,c}^{n-1}(z))^p$, we then obtain $h(z)\equiv \sum_{i=1}^{n-1}c^{p^{n-i}}$ (mod $p$); and so $h(z) + c \equiv \sum_{i=1}^{n-1}c^{p^{n-i}} + c$ (mod $p$). Now since $h(z) + c \equiv 0$ (mod $p$), it then follows $\sum_{i=1}^{n-1}c^{p^{n-i}} + c \equiv 0$ (mod $p$). But now observe $\sum_{i=1}^{n-1}c^{p^{n-i}} + c \equiv 0$ (mod $p$) can happen if $\sum_{i=1}^{n-1}c^{p^{n-i}}\equiv 0$ (mod $p$) and also $c\equiv 0$ (mod $p$); and so a contradiction. This then means  $f(x)=\varphi_{p,c}^n(x)-x$ has no roots in $\mathbb{Z}\slash p\mathbb{Z}$ for every $c\neq pt$ and for every fixed $n\in \mathbb{Z}_{\geq 2}$; and so we conclude $N_{c}^{(n)}(p) = 0$. This then completes the whole proof, as desired.
\end{proof}

\begin{rem}\label{rem2.3}
With now Theorem \ref{2.2}, we may to each $n$-periodic integral point of $\varphi_{p,c}$ associate $n$-periodic integral orbit. So now, a dynamical translation of Theorem \ref{2.2} is the claim that for any fixed (period) $n\geq 2$, the number of distinct $n$-periodic integral orbits that any $\varphi_{p,c}$ has when iterated on $\mathbb{Z} \slash p\mathbb{Z}$ is $p$ or zero. As we mentioned in Intro.\ref{sec1} that the count obtained in Theorem \ref{2.2} on the number of distinct $n$-periodic integral points of any $\varphi_{p,c}$ modulo $p$ may depend on $p$ (and so depend on deg$(\varphi_{p,c}))$ in one of the possibilities. Consequently, the expected total count (namely, $p+0 =p$) in Theorem \ref{2.2} on the number of distinct $n$-periodic integral points in the whole family of maps $\varphi_{p,c}$ modulo $p$ may not only depend on $p$, but also may grow to infinity as $p\to \infty$. 
\end{rem} 
\begin{rem}
Recall in [\cite{BK3}, Corollary 2.4] we proved that the function $N_{c}(p) = p$ or $0$ for every fixed prime $p\geq 3$ and for every coefficient $c$ divisible or indivisible by $p$. Now recall in the Proof of Theorem \ref{2.2} that for any fixed (period) $n\geq 2$, the counting function $N_{c}^{(n)}(p) = p$ or $0$ for every fixed prime $p$ and every $c$ divisible or indivisible by $p$. But now for any fixed (period) $n\geq 2$, we note that $N_{c}^{(n)}(p) = N_{c}(p) = p$ or $0$ for every fixed $p$ and for every $c$ divisible or indivisible by $p$. Moreover, for any fixed period $n\geq 2$ and for every coefficient $c$ divisible by $p$, we also note from [\cite{BK3}, Proof of Corollary 2.4] and from the Proof of Theorem \ref{2.2} that every $n$-periodic integral point of any $\varphi_{p,c}$ modulo $p$ is a fixed integral point. Consequently, we now note for every fixed period $n\geq 1$ and for every $c$ divisible by $p$, every $n$-periodic integral orbit of any $\varphi_{p,c}$ modulo $p$ is a fixed integral orbit; and moreover every $\varphi_{p,c}$ modulo $p$ may have $p$ distinct fixed integral orbits or no such points; a somewhat interesting precise arithmetic-geometric insight on all the $n$-periodic integral orbits of  $\varphi_{p,c}$ modulo $p$. 
\end{rem}

\section{On Number of $n$-Periodic Integral Points of any Family of Polynomial Maps $\varphi_{p-1,c}$}\label{sec3}

As in Section \ref{sec2}, we also wish to count the number of distinct $n$-periodic integral points of any $\varphi_{p-1,c}$ modulo $p$, where $p\geq 5$ is any given prime and $n\geq 2$ is any fixed integer. Again, we let  $p\geq 5$ be any prime, $c\in \mathbb{Z}$ be any integer and $n\geq 2$ be any fixed integer, and then define the following $n$-periodic point-counting function
\begin{equation}\label{M_{c}}
M_{c}^{(n)}(p) := \# \Biggl\{ z\in \mathbb{Z}\slash p\mathbb{Z}   : \begin{aligned} \varphi_{p-1,c}^{n-1}(z) -z \not \equiv 0 \ \text{(mod $p$)} \\ \ \varphi_{p-1,c}^{n}(z) - z \equiv 0 \ \text{(mod $p$)} \end{aligned} \Biggr\}.
\end{equation}\noindent Again, we first prove the following theorem on the number of distinct $n$-periodic points of any $\varphi_{4, c}$ modulo $5$:

\begin{thm} \label{3.1}
Let $\varphi_{4, c}$ be defined by $\varphi_{4, c}(z) = z^4 + c$ for all $c, z\in\mathbb{Z}$, and let $M_{c}^{(n)}(5)$ be defined as in \textnormal{(\ref{M_{c}})}. Then $M_{c}^{(n)}(5) = 1$ for every coefficient $c\equiv \pm 1 \ (mod \ 5)$ and fixed (even) integer $n$ or $M_{c}^{(n)}(5) = 2$ for every $c = 5t$; otherwise $M_{c}^{(n)}(5) = 0$ for any $c \equiv -1\ (mod \ 5)$ and fixed odd $n$ or $c\not \equiv \pm 1, 0 \ (mod \ 5)$ and fixed even $n$. 
\end{thm}
\begin{proof}
Let $g(z)= \varphi_{4,c}^n(z)-z = \varphi_{4,c}(\varphi_{4,c}^{n-1}(z)) - z = (\varphi_{4,c}^{n-1}(z))^4 - z + c$, and so $g(z)= (\varphi_{4,c}^{n-1}(z))^4 - z + c$. Now observe that applying the  multinomial theorem repeatedly on the term $(\varphi_{4,c}^{n-1}(z))^4$ right after applying the binomial theorem on the polynomial $(z^4 + c)^4$, we then obtain that $(\varphi_{4,c}^{n-1}(z))^4$ is a monic polynomial in $z$ of degree $4^n$ with integer coefficients in multiples of $c$. Thus, we may then write $(\varphi_{4,c}^{n-1}(z))^4 = z^{4^{n}} + h(z)$, where $h(z)$ is a non-constant polynomial in $z$ of deg$(h)<4^n$ with integer coefficients in multiples of $c$; and so we then obtain that $g(z)= z^{4^{n}} + h(z) - z + c$. Now for every coefficient $c=5t$, reducing both sides of $g(z)$ modulo $5$, we then obtain that $g(z)\equiv z^{4^n} - z$ (mod $5$), since we've also noted $h(z)\in c\mathbb{Z}[z]$ and so $h(z)\equiv 0$ (mod $5$); and so the reduced polynomial $g(z)$ modulo $5$ is now a polynomial defined over a finite field $\mathbb{Z}\slash 5\mathbb{Z}$ of $5$ distinct elements. So now, recall by Fermat's Little Theorem (FLT) that $z^4 = 1$ for every nonzero $z\in \mathbb{Z}\slash5\mathbb{Z}$, it then also follows $z^{4^n}= (z^4)^{4^{n-1}} = 1$ for every element $z\in (\mathbb{Z}\slash5\mathbb{Z})^{\times}=\mathbb{Z}\slash5\mathbb{Z}\setminus \{0\}$ and for every fixed $n\in \mathbb{Z}_{\geq 2}$. But then the reduced polynomial $g(z)\equiv 1 - z$ (mod $5$) for every point $z\in (\mathbb{Z}\slash5\mathbb{Z})^{\times}$ and so $g(z)$ has a root in $\mathbb{Z}\slash 5\mathbb{Z}$, namely, $z\equiv 1$ (mod $5$). Moreover, since $z$ is also a linear factor of $g(z)\equiv z(z^{4^n-1} - 1)$ (mod $5$) for every fixed integer $n\in \mathbb{Z}_{\geq 2}$, it then follows $z\equiv 0$ (mod $5$) is also root of $g(z)$ modulo $5$. But now, we then conclude $M_{c}^{(n)}(5) = 2$. To see $M_{c}^{(n)}(5) = 1$ for every coefficient $c\equiv 1$ (mod $5$) and for every fixed $n\in \mathbb{Z}_{\geq 2}$, we first note that writing $\varphi_{4,c}^{n-1}(z) = \underbrace{((((z^4 + c)^4 + c)^4 + c)^4 + \cdots + c)^4 + c}_\text{$(n-1)$ times}$ and then after each iteration, we reduce modulo $5$ along with the condition that $c\equiv 1$ (mod $5$) and also with $z^4 = 1$ for every $z\in (\mathbb{Z}\slash5\mathbb{Z})^{\times}$, it then follows $\varphi_{4,c}^{n-1}(z)\equiv 2$ (mod $5$) and $(\varphi_{4,c}^{n-1}(z))^4\equiv 1$ (mod $5$) for every fixed $n$. But then $g(z)=(\varphi_{4,c}^{n-1}(z))^4 - z + c\equiv 2-z$ (mod $5$) for every $z\in (\mathbb{Z}\slash5\mathbb{Z})^{\times}$ and so $g(z)$ modulo 5 has a root in $\mathbb{Z}\slash5\mathbb{Z}$, namely, $z\equiv 2$ (mod 5); from which we then conclude $M_{c}^{(n)}(5) = 1$. We now show $M_{c}^{(n)}(5) = 1$ for every coefficient $c\equiv -1$ (mod $5$) and fixed even integer $n\geq 2$. As before, we first note that since $c\equiv -1$ (mod $5$) and $z^4 = 1$ for every $z\in (\mathbb{Z}\slash5\mathbb{Z})^{\times}$, then reducing $\varphi_{4,c}^n(z)$ modulo $5$ after each iteration, we then obtain $\varphi_{4,c}^{n}(z)\equiv -1$ (mod $5$) for every fixed even $n\in \mathbb{Z}_{\geq 2}$. But then $g(z)= \varphi_{4,c}^n(z)-z \equiv -(1+z)$ (mod $5$) for every point $z\in (\mathbb{Z}\slash5\mathbb{Z})^{\times}$ and for every fixed even $n$ and thus $g(z)$ modulo $5$ has a root in $\mathbb{Z}\slash5\mathbb{Z}$, namely, $z\equiv -1$ (mod $5$); from which we then conclude $M_{c}^{(n)}(5) = 1$ as before. 

Finally, we now show $M_{c}^{(n)}(5) = 0$ for every coefficient $c \equiv -1$ (mod $5$) and for every fixed odd integer $n\geq 3$ or for every coefficient $c\not \equiv \pm1, 0$ (mod $5$) and every fixed even integer $n\geq 2$. Since $c\equiv -1$ (mod $5$) and also know $z^4 = 1$ for every $z\in (\mathbb{Z}\slash5\mathbb{Z})^{\times}$, then reducing $\varphi_{4,c}^n(z)$ modulo $5$ after each iteration, we then obtain that $\varphi_{4,c}^{n}(z)\equiv 0$ (mod $5$) for every fixed odd $n\in \mathbb{Z}_{\geq 3}$. But then $g(z)= \varphi_{4,c}^n(z)-z \equiv -z$ (mod $5$) for every $z\in (\mathbb{Z}\slash5\mathbb{Z})^{\times}$ and for every $c\equiv -1$ (mod $5$) and for every fixed odd $n$. But now observe $z\equiv 0$ (mod $5$) is a root of $g(z)$ modulo $5$ for every coefficient $c\equiv -1$ (mod $5$) and for every fixed odd $n$, and also $z\equiv 0$ (mod $5$) is also root of $g(z)$ modulo $5$ for every coefficient $c\equiv 0$ (mod $5$) and for every fixed odd $n$ as seen from the first part; from which we then obtain a contradiction that $-1 \equiv 0$ (mod $5$). Hence, it then follows $\varphi_{4,c}^n(x)-x$ has no roots in $\mathbb{Z} / 5\mathbb{Z}$ for every coefficient $c\equiv -1$ (mod $5$) and for every fixed odd $n\in \mathbb{Z}_{\geq 3}$; and so we conclude $M_{c}^{(n)}(5) = 0$. Otherwise, suppose $g(z)\equiv 0$ (mod $5$) for some $z\in (\mathbb{Z}\slash5\mathbb{Z})^{\times}$ and for every $c\not \equiv \pm1, 0$ (mod $5$) and every fixed even $n\in \mathbb{Z}_{\geq 2}$; and so $z^{4^n} + h(z)-z+c\equiv 0$ (mod $5$). So then, since $z^{4^n}=1$ for every $z\in (\mathbb{Z}\slash5\mathbb{Z})^{\times}$ and every fixed even $n$, we may write $z^{4^n} - z + h(z)+c\equiv 0$ (mod $5$) to obtain $1-z + h(z)+c\equiv 0$ (mod $5$). But now we note that $(1-z) + (h(z)+c)\equiv 0$ (mod $5$) can happen if $1-z\equiv 0$ (mod $5$) and also $h(z)+c\equiv 0$ (mod $5$). Moreover, recall from the first part that $1-z\equiv 0$ (mod $5$) when $c\equiv 0$ (mod $5$); which then contradicts $c\not \equiv \pm1, 0$ (mod $5$). This then means $g(x)=\varphi_{4,c}^n(x)-x$ has no roots in $\mathbb{Z} / 5\mathbb{Z}$ for every $c\not \equiv \pm1, 0$ (mod $5$) and every fixed even $n\in \mathbb{Z}_{\geq 2}$; and so we conclude $M_{c}^{(n)}(5) = 0$. This then completes the whole proof, as needed.
\end{proof} 
We now wish to generalize Theorem \ref{3.1} to any map $\varphi_{p-1, c}$ for any given prime $p\geq 5$. More precisely, we prove that the number of distinct $n$-periodic integral points of any map $\varphi_{p-1, c}$ modulo $p$ is also $1$ or $2$ or $0$:

\begin{thm} \label{3.2}
Let $p\geq 5$ be any fixed prime, and $\varphi_{p-1, c}$ be defined by $\varphi_{p-1, c}(z) = z^{p-1}+c$ for all $c, z\in\mathbb{Z}$. Let $M_{c}^{(n)}(p)$ be as in \textnormal{(\ref{M_{c}})}. Then $M_{c}^{(n)}(p) = 1$ for any $c\equiv \pm 1 \ (mod \ p)$ and fixed (even) $n$ or $M_{c}^{(n)}(p) = 2$ for any $c = pt$; otherwise $M_{c}^{(n)}(p) = 0$ for any $c \equiv -1\ (mod \ p)$ and fixed odd $n$ or $c\not \equiv \pm 1, 0 \ (mod \ p)$ and fixed even $n$.
\end{thm}
\begin{proof}
By applying a similar argument as in the Proof of Theorem \ref{3.1}, we then obtain the count as desired. That is, let $g(z)= \varphi_{p-1,c}^n(z)-z = \varphi_{p-1,c}(\varphi_{p-1,c}^{n-1}(z)) - z = (\varphi_{p-1,c}^{n-1}(z))^{p-1} - z + c$, and so $g(z)= (\varphi_{p-1,c}^{n-1}(z))^{p-1} - z + c$. Now as before, observe that applying the multinomial theorem repeatedly on $(\varphi_{p-1,c}^{n-1}(z))^{p-1}$ right after applying the binomial theorem on $(z^{p-1} + c)^{p-1}$, it then follows $(\varphi_{p-1,c}^{n-1}(z))^{p-1}$ is a monic polynomial in $z$ of degree $(p-1)^n$ with integer coefficients in multiples of $c$. Hence, we may then write $(\varphi_{p-1,c}^{n-1}(z))^{p-1} = z^{(p-1)^{n}} + h(z)$, where $h(z)$ is a non-constant polynomial in $z$ of deg$(h)<(p-1)^n$ with integer coefficients in multiples of $c$; and so we then write $g(z)= z^{(p-1)^{n}} + h(z) - z + c$. So now, for every coefficient $c=pt$, reducing $g(z)$ modulo $p$, it then follows $g(z)\equiv z^{(p-1)^n} - z$ (mod $p$), since also $h(z)\in c\mathbb{Z}[z]$ and thus $h(z)\equiv 0$ (mod $p$); and so the reduced polynomial $g(z)$ modulo $p$ is now a polynomial defined over a finite field $\mathbb{Z}\slash p\mathbb{Z}$ of order $p$. Now recall by (FLT) that $z^{p-1} = 1$ for every element $z\in (\mathbb{Z}\slash p\mathbb{Z})^{\times}=\mathbb{Z}\slash p\mathbb{Z}\setminus \{0\}$, it then also follows $(z^{p-1})^{(p-1)^{n-1}}= 1$ for every element $z\in (\mathbb{Z}\slash p\mathbb{Z})^{\times}$ and for every fixed $n\in \mathbb{Z}_{\geq 2}$. But then $g(z)\equiv 1 - z$ (mod $p$) for every point $z\in (\mathbb{Z}\slash p\mathbb{Z})^{\times}$, and so $g(z)$ has a root in $\mathbb{Z}\slash p\mathbb{Z}$, namely, $z\equiv 1$ (mod $p$). Moreover, since $z$ is also a linear factor of $g(z)\equiv z(z^{(p-1)^n-1} - 1)$ (mod $p$), then $z\equiv 0$ (mod $p$) is also root of $g(z)$ modulo $p$. But now, we then conclude $M_{c}^{(n)}(p) = 2$. We now show $M_{c}^{(n)}(p) = 1$ for every coefficient $c\equiv 1$ (mod $p$) and for every fixed $n\in \mathbb{Z}_{\geq 2}$. Note that writing $\varphi_{p-1,c}^{n-1}(z) = \underbrace{((((z^{p-1} + c)^{p-1} + c)^{p-1} + c)^{p-1} + \cdots + c)^{p-1} + c}_\text{$(n-1)$ times}$ and then after each iteration, we reduce modulo $p$ along with $c\equiv 1$ (mod $p$) and $z^{p-1} = 1$ for every $z\in (\mathbb{Z}\slash p\mathbb{Z})^{\times}$, it then follows $\varphi_{p-1,c}^{n-1}(z)\equiv 2$ (mod $p$) and $(\varphi_{p-1,c}^{n-1}(z))^{p-1}\equiv 1$ (mod $p$) for every fixed $n$. But then $g(z)=(\varphi_{p-1,c}^{n-1}(z))^{p-1} - z + c\equiv 2-z$ (mod $p$) for every $z\in (\mathbb{Z}\slash p\mathbb{Z})^{\times}$ and for every fixed $n$; from which it then follows that $g(z)$ modulo $p$ has a root in $\mathbb{Z}\slash p\mathbb{Z}$, namely, $z\equiv 2$ (mod $p$); and so we conclude $M_{c}^{(n)}(p) = 1$. We now show $M_{c}^{(n)}(p) = 1$ for every coefficient $c\equiv -1$ (mod $p$) and for every fixed even $n\in \mathbb{Z}_{\geq 2}$. As before, since $c\equiv -1$ (mod $p$) and also $z^{p-1} = 1$ for every $z\in (\mathbb{Z}\slash p\mathbb{Z})^{\times}$, reducing $\varphi_{p-1,c}^n(z)$ modulo $p$ after each iteration, it then follows $\varphi_{p-1,c}^{n}(z)\equiv -1$ (mod $p$) for every fixed even $n$. But then $g(z)= \varphi_{p-1,c}^n(z)-z \equiv -(1+z)$ (mod $p$) for every point $z\in (\mathbb{Z}\slash p\mathbb{Z})^{\times}$; and so $g(z)$ modulo $p$ has a root in $\mathbb{Z}\slash p\mathbb{Z}$, namely, $z\equiv -1$ (mod $p$). But now, as before we then conclude $M_{c}^{(n)}(p) = 1$. 

Finally, we now show $M_{c}^{(n)}(p) = 0$ for every coefficient $c \equiv -1$ (mod $p$) and every fixed odd $n\in \mathbb{Z}_{\geq 3}$ or every coefficient $c\not \equiv \pm1, 0$ (mod $p$) and every fixed even $n\in \mathbb{Z}_{\geq 2}$. As before, since $c\equiv -1$ (mod $p$) and $z^{p-1} = 1$ for every $z\in (\mathbb{Z}\slash p\mathbb{Z})^{\times}$, reducing $\varphi_{p-1,c}^n(z)$ modulo $p$ after each iteration, it then follows $\varphi_{p-1,c}^{n}(z)\equiv 0$ (mod $p$) for every fixed odd $n$. But then $g(z)= \varphi_{p-1,c}^n(z)-z \equiv -z$ (mod $p$) for every $z\in (\mathbb{Z}\slash p\mathbb{Z})^{\times}$ and for every $c \equiv -1$ (mod $p$) and for every fixed odd $n$. But now $z\equiv 0$ (mod $p$) is a root of $g(z)$ modulo $p$ for every $c\equiv -1$ (mod $p$) and every fixed odd $n$, and also $z\equiv 0$ (mod $p$) is also root of $g(z)$ modulo $p$ for every $c\equiv 0$ (mod $p$) and every fixed odd $n$ as seen from the first part; from which we then obtain a contradiction that $-1 \equiv 0$ (mod $p$). Thus, it then follows $\varphi_{p-1,c}^n(x)-x$ has no roots in $\mathbb{Z} / p\mathbb{Z}$ for every $c\equiv -1$ (mod $p$) and every fixed odd $n$; and so we conclude $M_{c}^{(n)}(p) = 0$. Otherwise, if $g(z)\equiv 0$ (mod $p$) for some $z\in (\mathbb{Z}\slash p\mathbb{Z})^{\times}$ and every $c\not \equiv \pm1, 0$ (mod $p$) and fixed even $n\in \mathbb{Z}_{\geq 2}$; and so $z^{(p-1)^n} + h(z)-z+c\equiv 0$ (mod $p$). So then, since $z^{(p-1)^n}=1$ for every $z\in \mathbb{Z}\slash p\mathbb{Z}\setminus \{0\}$ and fixed even $n$, then $z^{(p-1)^n} -z + h(z) +c\equiv 0$ (mod $p$) becomes $1-z + h(z)+c\equiv 0$ (mod $p$). But now we note that $(1-z) + (h(z)+c)\equiv 0$ (mod $p$) can happen if $1-z\equiv 0$ (mod $p$) and also $h(z)+c\equiv 0$ (mod $p$). Moreover, recall from the first part that $1-z\equiv 0$ (mod $p$) when $c\equiv 0$ (mod $p$); which then contradicts $c\not \equiv \pm1, 0$ (mod $p$). It then follows $g(x)=\varphi_{p-1,c}^n(x)-x$ has no roots in $\mathbb{Z} / p\mathbb{Z}$ for every $c\not \equiv \pm1, 0$ (mod $p$) and every fixed even $n\in \mathbb{Z}_{\geq 2}$; and so we conclude $M_{c}^{(n)}(p) = 0$. This then completes the whole proof, as needed.
\end{proof}

\begin{rem}
With now Theorem \ref{3.2}, we then to each distinct $n$-periodic integral point of $\varphi_{p-1,c}$ associate $n$-periodic integral orbit. In doing so, we then also obtain a dynamical translation of Theorem \ref{3.2}, namely, that for any fixed (period) $n\geq 2$, the number of distinct $n$-periodic integral orbits that any $\varphi_{p-1,c}$ has when iterated on $\mathbb{Z} / p\mathbb{Z}$ is equal to $1$ or $2$ or $0$. As noted in Intro.\ref{sec1} that the count obtained in Theorem \ref{3.2} on the number of distinct $n$-periodic integral points of any $\varphi_{p-1,c}$ modulo $p$ is independent of $p$ (and so independent of deg$(\varphi_{p-1,c}))$ in each of the possibilities. Moreover, we may also observe that the expected total count (namely, $1 + 2 + 0 =3$ for any fixed odd period $n\in \mathbb{Z}_{\geq 3}$ or $1 + 1 + 2 + 0 =4$ for any fixed even period $n\in \mathbb{Z}_{\geq 2}$) of the number of distinct $n$-periodic integral points (orbits) in the whole family of maps $\varphi_{p-1,c}$ modulo $p$ is not only also independent of $p$ (and so independent of deg$(\varphi_{p-1,c})$), but is also equal to $3$ or $4$ as $p-1\to \infty$; a somewhat interesting phenomenon differing significantly from a phenomenon that we remarked about in \ref{rem2.3}, however, coinciding somewhat with an observation that we made in [\cite{BK1}, Remark 3.3]. Additionally, notice when $n=2$, the expected total count in Theorem \ref{3.2} on the number of distinct $2$-periodic integral points of any $\varphi_{p-1,c}$ modulo $p$ is equal to $1 + 1 + 2 + 0 =4$; coinciding somewhat with an observation noted in Remark \ref{ht}. Moreover, recall in [\cite{BK1}, Proof of Theorem 3.2] and in Proof of Theorem \ref{3.2} that $z\equiv 1, 0, 2$ (mod $p$) are fixed and also $2$-periodic integral points of $\varphi_{p-1,c}$ modulo $p$. It may then follow from that the expected total number of distinct fixed and $2$-periodic integral points in the whole family of maps $\varphi_{p-1,c}$ modulo $p$ is 4; a somewhat interesting observation coinciding with the first predicted upper bound 4 observed in Remark \ref{ht} on Conj.\ref{conjecture 3.2.1}.
\end{rem}

\begin{rem}
Recall in [\cite{BK1}, Section 3] we proved that the fixed point-counting function $M_{c}(p) = 1, 2$ or $0$ for every fixed prime $p$ and for every coefficient $c\equiv 1, 0$ (mod $p$) or $c\equiv -1$ (mod $p$). Now recall in proving Theorem \ref{3.2}, we found that for any fixed odd period $n\geq 3$, $z\equiv 1, 0, 2$ (mod $p$) are $n$-periodic integral points of any $\varphi_{p-1,c}$ modulo $p$ (which mind you also showed up as fixed integral points of $\varphi_{p-1,c}$ modulo $p$ in \cite{BK1}); from which we then also concluded that the function $M_{c}^{(n)}(p)  = 1, 2$ or $0$ for every fixed prime $p$ and every coefficient $c\equiv 1, 0$ (mod $p$) or $c\equiv -1$ (mod $p$). But now for every fixed odd (period) $n\geq 3$, we note that $M_{c}^{(n)}(p) = M_{c}(p) = 1, 2$ or $0$ for every fixed prime $p$ and for every $c\equiv 1, 0$ (mod $p$) or $c\equiv -1$ (mod $p$). Consequently, we now note that for any fixed odd period $n\geq 1$, every $n$-periodic integral orbit of $\varphi_{p-1,c}$ modulo $p$ is a fixed integral orbit; and moreover every $\varphi_{p-1,c}$ modulo $p$ may have one or two or no fixed integral orbits; a somewhat interesting precise arithmetic-geometric insight on all odd $n$-periodic integral orbits of any $\varphi_{p-1,c}$ modulo $p$. Furthermore, recall in proving Theorem \ref{3.2}, we also found that for any fixed even period $n\geq 2$, the points $z\equiv 1, 0, 2, -1$ (mod $p$) are $n$-periodic integral points of any $\varphi_{p-1,c}$ modulo $p$; from which we then concluded that the function $M_{c}^{(n)}(p)  = 1, 2$ or $0$ for every fixed prime $p$ and for every $c\equiv \pm 1, 0$ (mod $p$) or $c\not \equiv \pm 1, 0$ (mod $p$). But now for every fixed even (period) $n\geq 4$, we also note that the counting function $M_{c}^{(n)}(p) = M_{c}^{(2)}(p) = 1, 2$ or $0$ for every fixed prime $p$ and every $c\equiv \pm 1, 0$ (mod $p$) or $c\not \equiv \pm 1, 0$ (mod $p$). As before, we now also note that for any fixed even period $n\geq 2$, every $n$-periodic integral orbit of $\varphi_{p-1,c}$ modulo $p$ is a $2$-periodic integral orbit; and moreover every $\varphi_{p-1,c}$ modulo $p$ may have one or two or no $2$-periodic integral orbits; another somewhat interesting precise arithmetic-geometric insight on all even $n$-periodic integral orbits of every $\varphi_{p-1,c}$ modulo $p$.    
\end{rem}

\section{Dynamical Complexity of Forward Periodic Orbit Structure of any $\varphi_{p,c}$ and $\varphi_{p-1,c}$}\label{sec4a}
Observe in Theorem \ref{2.2} that the number $N_{c}^{(n)}(p)$ is independent of (period) $n$, and so we may then have 
\begin{center}
    $\lim\limits_{n\to \infty} N_{c}^{(n)}(p) = p$ or $0$.
\end{center}

Now as we mentioned in the first opening article \cite{BK1} of this ongoing multi-part series, that one of the main objectives in classical dynamical systems is to understand \textit{all} orbits usually via topological and analytic techniques. In doing so, one may not only find that orbits can easily get very complicated, but also one may easily find that important and interesting statistical questions (for instance, determining \say{\textit{topological entropy}}) concerning measuring the complexity of the underlying system, may become intractable. (Recall loosely speaking that \say{\textit{topological entropy}} is a nonnegative statistic that measures the exponential growth rate of distinguishable orbits of a dynamical system as time progresses. The interested reader may read in more great details about topological entropy in work of Adler \cite{Adl} and Bowen \cite{Ruf}). So now, motivated by such a statistic and moreover since we may also now view $N_{c}^{(n)}(p)$ as now the $n$-periodic orbit-counting function, we in this section  wish to study very mildly the complexity of our polynomial discrete dynamical system $(\mathbb{Z}\slash p\mathbb{Z}$, $\varphi_{p,c}$ modulo $p$). We do so via analyzing the behavior of a classically important statistic from dynamical systems, namely, the exponential growth rate $\rho(\varphi_{p,c})$ [\cite{Kat}, Page 196]; and which then shows $N_{c}^{(n)}(p)$ grows by a factor $1=e^{\rho(\varphi_{p,c})}$ as period $n\to \infty$ and so the $n$-periodic orbit-counting function $N_{c}^{(n)}(p)$ is indeed a constant function. That is, we have: 
\begin{cor}\label{4.1}
Assume Theorem \ref{2.2}, and let $n\geq 2$ be any period. Then the exponential growth rate of $n$-periodic orbit-counting function $N_{c}^{(n)}(p)$ exists and is equal to zero. More precisely, we have 

\begin{center}
    $\rho(\varphi_{p,c}) := \limsup\limits_{n\to \infty}\frac{\text{log }(\text{max} \{N_{c}^{(n)}(p),1\})}{n} = 0$.
\end{center} 
\end{cor}
\begin{proof}
Since we know from Theorem \ref{2.2} that the number $N_{c}^{(n)}(p) = p\text{ or } 0$ for every fixed (period) $n\geq 2$, we then obtain $\frac{\text{log }(\text{max} \{N_{c}^{(n)}(p),1\})}{n} = \frac{\text{log }p}{n}$ or $0$. But now letting $n\to \infty$, we then obtain $\rho(\varphi_{p,c}) = 0$, as desired.
\end{proof}

Similarly, we may also observe in Theorem \ref{3.2} that $M_{c}^{(n)}(p)$ is also independent of (period) $n$, and so  
\begin{center}
    $\lim\limits_{n\to \infty} M_{c}^{(n)}(p) = 1, 2 \text{ or } 0.$
\end{center}As before, we also wish to study very mildly the complexity of the system $(\mathbb{Z}\slash p\mathbb{Z}$, $\varphi_{p-1,c}$ modulo $p$) by again determining the behavior of the associated exponential growth rate $\rho(\varphi_{p-1,c})$. With that mind, we then also obtain $M_{c}^{(n)}(p)$ grows by a factor $1=e^{\rho(\varphi_{p-1,c})}$ as $n\to \infty$ and so $M_{c}^{(n)}(p)$ is also a constant function. That is:
\begin{cor}
Assume Theorem \ref{3.2}, and let $n\geq 2$ be any period. Then the exponential growth rate of $n$-periodic orbit-counting function $M_{c}^{(n)}(p)$ exists and is equal to zero. More precisely, we have 
\begin{center}
    $\rho(\varphi_{p-1,c}) := \limsup\limits_{n\to \infty}\frac{\text{log }(\text{max} \{M_{c}^{(n)}(p),1\})}{n} = 0$.
\end{center}
\end{cor}
\begin{proof}
Applying a similar argument as in the Proof of Corollary \ref{4.1}, we then obtain $\rho(\varphi_{p-1,c})=0$, as desired.
\end{proof}

\section{On the Average Number of $n$-Periodic Points of any Polynomial Map $\varphi_{p,c}$ \& $\varphi_{p-1,c}$ }\label{sec4}

In this section, we wish to inspect independently the behavior of the counting functions $N_{c}^{(n)}(p)$ and $M_{c}^{(n)}(p)$ as $c\to \infty$. First, we wish to determine: \say{\textit{For any fixed (period) $n\in \mathbb{Z}_{\geq 2}$, what is the average value of $N_{c}^{(n)}(p)$ as $c \to \infty$?}} The following corollary shows that the average value of $N_{c}^{(n)}(p)$ is zero or unbounded as $c\to \infty$:
\begin{cor}\label{cor5}
Let $p\geq 3$ be any prime integer, and let $n\geq 2$ be any fixed (period). Then the average value of $n$-periodic point-counting function $N_{c}^{(n)}(p)$ is either zero or unbounded as $c\to\infty$. \textit{More precisely, we have} 
\begin{myitemize}
    \item[\textnormal{(a)}] \textnormal{Avg} $N_{c \neq pt}^{(n)}(p) := \lim\limits_{c\to\infty} \Large{\frac{\sum\limits_{3\leq p\leq c, \ p\nmid c}N_{c}^{(n)}(p)}{\Large{\sum\limits_{3\leq p\leq c, \ p\nmid c}1}}} = 0.$  
    
    \item[\textnormal{(b)}] \textnormal{Avg} $N_{c = pt}^{(n)}(p):= \lim\limits_{c \to\infty} \Large{\frac{\sum\limits_{3\leq p\leq c, \ p\mid c}N_{c}^{(n)}(p)}{\Large{\sum\limits_{3\leq p\leq c, \ p\mid c}1}}} =  \infty$.    
\end{myitemize}

\end{cor}
\begin{proof}
By applying a similar argument as in [\cite{BK3}, Proof of Cor. 7.1], we then obtain the limits, as desired.
\end{proof} 
\begin{rem} \label{re5}
From arithmetic statistics to arithmetic dynamics, we note that Corollary \ref{cor5} shows that any $\varphi_{p,c}$ iterated on the space $\mathbb{Z} / p\mathbb{Z}$ has on average zero or unbounded number of $n$-periodic integral orbits modulo $p$ as $c\to \infty$; a somewhat interesting averaging phenomenon coinciding precisely with an averaging phenomenon in [\cite{BK3}, Remark 7.2] on the average number of distinct fixed integral points of every map $\varphi_{p, c}$ iterated on $\mathbb{Z} / p\mathbb{Z}$.
\end{rem}

Similarly, we also wish to determine: \say{\textit{For any fixed (period) $n\in \mathbb{Z}_{\geq 2}$, what is the average value of $M_{c}^{(n)}(p)$ as $c \to \infty$?}} The following corollary shows that the average value of $M_{c}^{(n)}(p)$ is $1$ or $2$ or $0$ as $c\to \infty$:

\begin{cor}\label{cor4.3}
Let $p\geq 5$ be any prime integer, and let $n\geq 2$ be any fixed (period). Then the average value of $n$-periodic point-counting function $M_{c}^{(n)}(p)$ exists and is $1$ or $2$ or $0$ as $c\to\infty$. More precisely, we have 
\begin{myitemize}
    \item[\textnormal{(a)}] \textnormal{Avg} $M_{c\pm1 = pt}^{(n)}(p) := \lim\limits_{c\to\infty} \Large{\frac{\sum\limits_{5\leq p\leq c, \ p\mid (c\pm1)}M_{c}^{(n)}(p)}{\Large{\sum\limits_{5\leq p\leq c, \ p\mid (c\pm1)}1}}} = 1.$ 

    \item[\textnormal{(b)}] \textnormal{Avg} $M_{c= pt}^{(n)}(p) := \lim\limits_{c\to\infty} \Large{\frac{\sum\limits_{5\leq p\leq c, \ p\mid c}M_{c}^{(n)}(p)}{\Large{\sum\limits_{5\leq p\leq c, \ p\mid c}1}}} = 2.$
    
     \item[\textnormal{(c)}] \textnormal{Avg} $M_{c+1=pt, n =2k+1 \text{ or }c\not \equiv\pm1, 0 \ (\text{mod }p)}^{(n)}(p):= \lim\limits_{c \to\infty} \Large{\frac{\sum\limits_{5\leq p\leq c, \ c+1=pt, n =2k+1 \text{ or }c\not \equiv\pm1, 0 \ (\text{mod }p)}M_{c}^{(n)}(p)}{\Large{\sum\limits_{5\leq p\leq c, \ c+1=pt, n =2k+1 \text{ or }c\not \equiv\pm1, 0 \ (\text{mod }p)}1}}} =  0$.
\end{myitemize}

\end{cor}
\begin{proof}
By applying a similar argument as in [\cite{BK1}, Proof of Cor. 4.3], we then obtain the limits, as desired.
\end{proof} 
\begin{rem} \label{re5.4}
As before, we again also note that from arithmetic statistics to arithmetic dynamics, Corollary \ref{cor4.3} shows that any $\varphi_{p-1,c}$ iterated on the space $\mathbb{Z} / p\mathbb{Z}$ has on average one or two or no $n$-periodic integral orbits as $c\to \infty$; a somewhat interesting averaging phenomenon coinciding with an averaging phenomenon in [\cite{BK1}, Remark 4.4] on the average number of distinct fixed integral points of any $\varphi_{p-1, c}$ iterated on the space $\mathbb{Z} / p\mathbb{Z}$.
\end{rem}

\section{On the Density of Monic Integer Polynomials $\varphi_{p,c}(x)$ with the Number $N^{(n)}_{c}(p) = p$}\label{sec5}

As in [\cite{BK3}, Section 9] we in section also wish to determine: \say{\textit{For a prime $p\geq 3$ and for any fixed period $n\in \mathbb{Z}_{\geq 2}$, what is the density of monic integer polynomials $\varphi_{p, c}(x) = x^p + c$ with exactly $p$ distinct $n$-periodic integral points modulo $p$?}} The following immediate corollary shows that for any fixed period $n\in \mathbb{Z}_{\geq 2}$, there are (as in \cite{BK3}) very few monic integer polynomials $\varphi_{p,c}(x) = x^p + c$ having only $p$ distinct $n$-periodic integral points modulo $p$:
\begin{cor}\label{5.1}
Let $p\geq 3$ be a prime integer. Then the density of monic polynomials $\varphi_{p,c}(x) = x^p + c\in \mathbb{Z}[x]$ having the number $N_{c}^{(n)}(p) = p$ exists and is equal to $0 \%$ as $c\to \infty$. More precisely, we have
\begin{center}
    $\lim\limits_{c\to\infty} \Large{\frac{\# \{\varphi_{p,c}(x) \in \mathbb{Z}[x]\ : \ 3\leq p\leq c \ and \ N_{c}^{(n)}(p) \ = \ p\}}{\Large{\# \{\varphi_{p,c}(x) \in \mathbb{Z}[x]\ : \ 3\leq p\leq c \}}}} = \ 0.$
\end{center}
\end{cor}
\begin{proof}
Since the defining condition $N_{c}^{(n)}(p) = p$ is as we proved in Theorem \ref{2.2}, determined whenever the coefficient $c$ is divisible by any prime $p\geq 3$, we may then count $\# \{\varphi_{p,c}(x) \in \mathbb{Z}[x] : 3\leq p\leq c \ \text{and} \ N_{c}^{(n)}(p) \ = \ p\}$ by counting the number $\# \{\varphi_{p,c}(x)\in \mathbb{Z}[x] : 3\leq p\leq c \ \text{and} \ p\mid c \ \text{for \ any \ fixed} \ c \}$. In that case, we then write 
\begin{center}
$\Large{\frac{\# \{\varphi_{p,c}(x) \in \mathbb{Z}[x] \ : \ 3\leq p\leq c \ \text{and} \ N_{c}^{(n)}(p) \ = \ p\}}{\Large{\# \{\varphi_{p,c}(x) \in \mathbb{Z}[x] \ : \ 3\leq p\leq c \}}}} = \Large{\frac{\# \{\varphi_{p,c}(x)\in \mathbb{Z}[x] \ : \ 3\leq p\leq c \ \text{and} \ p\mid c \ \text{for any fixed} \ c \}}{\Large{\# \{\varphi_{p,c}(x) \in \mathbb{Z}[x] \ : \ 3\leq p\leq c \}}}}$. 
\end{center}\indent Moreover, for any fixed integer $c\geq 3$, the numerator of the foregoing quotient may be rewritten to then obtain
\begin{center}
$\# \{\varphi_{p,c}(x) \in \mathbb{Z}[x] : 3\leq p\leq c \ \text{and} \ p\mid c \} = \# \{p : 3\leq p\leq c \text{ and } p\mid c \} = \sum_{3\leq p\leq c, \ p\mid c}1 = \omega (c)$, 
\end{center}where $\omega(n)$ is by definition the number of distinct prime factors of $n$. Writing $\# \{\varphi_{p,c}(x) \in \mathbb{Z}[x]  : 3\leq p\leq c \} = \sum_{3\leq p\leq c} 1 = \pi(c)$, where $\pi(n)$ is by definition the number of primes at most $n$, we then note that the quotient 
\begin{center}
$\Large{\frac{\# \{\varphi_{p,c}(x)\in \mathbb{Z}[x] \ : \ 3\leq p\leq c \ \text{and} \ p\mid c \ \text{for any fixed} \ c \}}{\Large{\# \{\varphi_{p,c}(x)\in \mathbb{Z}[x] \ : \ 3\leq p\leq c \}}}} = \frac{\omega(c)}{\pi(c)}$.
\end{center}So now, recall (from a well-known fact) that for any $c\in \mathbb{Z}_{\geq 3}$, we have $2^{\omega(c)}\leq \sigma (n) \leq 2^{\Omega(c)}$, where $\sigma(n)$ is by definition the divisor function and $\Omega(n)$ is by definition the total number of prime factors of $n$, with respect to their multiplicity. Note that taking logarithms, we then obtain $\omega(c)\leq \frac{\text{log} \ \sigma(c)}{\text{log} \ 2}$; and so $\frac{\omega(c)}{\pi(c)} \leq \frac{\text{log} \ \sigma(c)}{\text{log} \ 2 \cdot \pi(c)}$. Moreover, for every $\epsilon >0$, it is well-known that $\sigma(c) = o(c^{\epsilon})$; and so log $\sigma(c) =$ log $o(c^{\epsilon})$ and so have $\frac{\omega(c)}{\pi(c)} \leq \frac{\text{log} \ o(c^{\epsilon})}{\text{log} \ 2 \cdot \pi(c)}$. Now for every fixed $\epsilon>0$, we then note $\lim\limits_{c\to\infty} \frac{\text{log} \ o(c^{\epsilon})}{\text{log} \ 2 \cdot \pi(c)} = 0$ and so $\lim\limits_{c\to\infty} \frac{\omega(c)}{\pi(c)} \leq 0$. But now 
\begin{center}
$\lim\limits_{c\to\infty} \Large{\frac{\# \{\varphi_{p,c}(x)\in \mathbb{Z}[x] \ : \ 3\leq p\leq c \ \text{and} \ N_{c}^{(n)}(p) \ = \ p\}}{\Large{\# \{\varphi_{p,c}(x) \in \mathbb{Z}[x] \ : \ 3\leq p\leq c \}}}} =\lim\limits_{c\to\infty} \frac{\omega(c)}{\pi(c)} \leq 0$.
\end{center}Moreover, we also observe that the number $\# \{\varphi_{p,c}(x)\in \mathbb{Z}[x] : 3\leq p\leq c \ \text{and} \ N_{c}^{(n)}(p) \ = \ p\}\geq 1$, and so have 
\begin{center}
$\lim\limits_{c\to\infty}\Large{\frac{\# \{\varphi_{p,c}(x) \in \mathbb{Z}[x] \ : \ 3\leq p\leq c \ \text{and} \ N_{c}^{(n)}(p) \ = \ p\}}{\Large{\# \{\varphi_{p,c}(x) \in \mathbb{Z}[x] \ : \ 3\leq p\leq c \}}}}\geq \lim\limits_{c\to\infty}\frac{1}{\pi(c)} = 0$. But now, we then conclude that the limit 
\end{center}  $\lim\limits_{c\to\infty} \Large{\frac{\# \{\varphi_{p,c}(x) \in \mathbb{Z}[x] \ : \ 3\leq p\leq c \ \text{and} \ N_{c}^{(n)}(p) \ = \ p\}}{\Large{\# \{\varphi_{p,c}(x) \in \mathbb{Z}[x] \ : \ 3\leq p\leq c \}}}} = 0$ as needed. This then completes the whole proof, as desired.
\end{proof}\noindent Note that we may also interpret Corollary \ref{5.1} as saying that for any fixed period $n\in \mathbb{Z}_{\geq 2}$, the probability of choosing randomly a polynomial $\varphi_{p,c}(x)\in \mathbb{Z}[x]$ with $p$ distinct $n$-periodic integral points modulo $p$ is equal to zero; a somewhat interesting probabilistic phenomenon coinciding with a phenomenon remarked in [\cite{BK3}, Sect.9] on the probability of choosing randomly a monic $\varphi_{p,c}(x)\in \mathbb{Z}[x]$ with $p$ distinct fixed integral points modulo $p$.

\section{On Densities of Monic Integer Polynomials $\varphi_{p-1,c}(x)$ with $M_{c}^{(n)}(p) = 1$ \& $M_{c}^{(n)}(p) = 2$}\label{sec6}

As in Section \ref{sec5}, we in this section also wish to determine: \say{\textit{For a prime $p\geq 5$ and for any fixed period $n\geq 2$, what is the density of monic integer polynomials $\varphi_{p-1,c}(x) = x^{p-1}+c$ with exactly two distinct $n$-periodic integral points modulo $p$?}} The following corollary shows that for any fixed period $n\geq 2$, the density of monic integer polynomials $\varphi_{p-1,c}(x)=x^{p-1}+ c$ with exactly two distinct $n$-periodic integral points modulo $p$ is zero:
\begin{cor}\label{6.1}
Let $p\geq 5$ be a prime integer. The density of monic polynomials $\varphi_{p-1,c}(x)=x^{p-1}+c\in \mathbb{Z}[x]$ having the number $M_{c}^{(n)}(p) = 2$ exists and is equal to $0\%$ as $c\to \infty$. More precisely, we have 
\begin{center}
    $\lim\limits_{c\to\infty} \Large{\frac{\# \{\varphi_{p-1,c}(x) \in \mathbb{Z}[x]\ : \ 5\leq p\leq c \ and \ M_{c}^{(n)}(p) \ = \ 2\}}{\Large{\# \{\varphi_{p-1,c}(x) \in \mathbb{Z}[x]\ : \ 5\leq p\leq c \}}}} = \ 0.$
\end{center}
\end{cor}

\begin{proof}
As before, since $M_{c}^{(n)}(p) = 2$ is as we proved in Theorem \ref{3.2}, determined whenever the coefficient $c$ is divisible by $p$, hence, we may then count the number $\# \{\varphi_{p-1,c}(x) \in \mathbb{Z}[x] : 5\leq p\leq c \ \text{and} \ M_{c}^{(n)}(p) \ = \ 2\}$ by simply counting $\# \{\varphi_{p-1,c}(x)\in \mathbb{Z}[x] : 5\leq p\leq c \ \text{and} \ p\mid c \ \text{for \ any \ fixed} \ c \}$. But now applying a similar argument as in the Proof of Corollary \ref{5.1}, we then obtain that the limit exists and is equal to 0, as desired.
\end{proof} \noindent As before, we may also interpret Corollary \ref{6.1} as saying that for any fixed period $n\in \mathbb{Z}_{\geq 2}$, the probability of choosing randomly a monic polynomial $\varphi_{p-1,c}(x)\in \mathbb{Z}[x]$ with two $n$-periodic integral points modulo $p$ is zero; a somewhat interesting probabilistic phenomenon coinciding with a phenomenon remarked in [\cite{BK1}, Section 6] on the probability of choosing randomly a monic $\varphi_{p-1,c}(x)\in \mathbb{Z}[x]$ with two distinct fixed integral points modulo $p$.

The following corollary shows that for any fixed period $n\in \mathbb{Z}_{\geq 2}$, the probability of choosing randomly a monic polynomial $\varphi_{p-1,c}(x)\in \mathbb{Z}[x]$ having exactly one $n$-periodic integral point modulo $p$ is also equal to zero:

\begin{cor}\label{6.2}
Let $p\geq 5$ be a prime integer. The density of monic polynomials $\varphi_{p-1,c}(x)=x^{p-1}+c\in \mathbb{Z}[x]$ having the number $M_{c}^{(n)}(p) = 1$ exists and is equal to $0\%$ as $c\to \infty$. More precisely, we have  
\begin{center}
    $\lim\limits_{c\to\infty} \Large{\frac{\# \{\varphi_{p-1,c}(x) \in \mathbb{Z}[x]\ : \ 5\leq p\leq c \ and \ M_{c}^{(n)}(p) \ = \ 1\}}{\Large{\# \{\varphi_{p-1,c}(x) \in \mathbb{Z}[x]\ : \ 5\leq p\leq c \}}}} = \ 0.$
\end{center}
\end{cor}
\begin{proof}
As before, $M_{c}^{(n)}(p) = 1$ is as we proved in Theorem \ref{3.2}, determined whenever the coefficient $c$ is such that $c\pm1$ is divisible by any fixed prime $p\geq 5$; and so we may count $\# \{\varphi_{p-1,c}(x) \in \mathbb{Z}[x] : 5\leq p\leq c \ \text{and} \ M_{c}^{(n)}(p) \ = \ 1\}$ by simply counting the number $\# \{\varphi_{p-1,c}(x)\in \mathbb{Z}[x] : 5\leq p\leq c \ \text{and} \ p\mid (c\pm1) \ \text{for \ any \ fixed} \ c \}$. But now, since $c-1<c$, then if the number $\# \{p : 5\leq p\leq c \ \text{and} \ p\mid (c-1) \}< \# \{p : 5\leq p\leq c \ \text{and} \ p\mid c \}$, we then obtain that
\begin{center}
$\Large{\frac{\# \{\varphi_{p-1,c}(x) \in \mathbb{Z}[x] \ : \ 5\leq p\leq c \ \text{and} \ p\mid (c-1) \ \text{for any fixed} \ c\}}{\Large{\# \{\varphi_{p-1,c}(x) \in \mathbb{Z}[x] \ : \ 5\leq p\leq c \}}}} < \Large{\frac{\# \{\varphi_{p-1,c}(x)\in \mathbb{Z}[x] \ : \ 5\leq p\leq c \ \text{and} \ p\mid c \ \text{for any fixed} \ c \}}{\Large{\# \{\varphi_{p-1,c}(x) \in \mathbb{Z}[x] \ : \ 5\leq p\leq c \}}}}.$ 
\end{center}But now applying a similar argument as in [\cite{BK1}, Proof of Cor. 6.2], we then obtain that the limit is equal to $0$ as desired in this case. Otherwise, if the size $\# \{p : 5\leq p\leq c \ \text{and} \ p\mid c \}< \# \{p : 5\leq p\leq c \ \text{and} \ p\mid (c-1) \}$, then
\begin{center}
$\Large{\frac{\# \{\varphi_{p-1,c}(x) \in \mathbb{Z}[x] \ : \ 5\leq p\leq c \ \text{and} \ p\mid c \ \text{for any fixed} \ c\}}{\Large{\# \{\varphi_{p-1,c}(x) \in \mathbb{Z}[x] \ : \ 5\leq p\leq c \}}}} < \Large{\frac{\# \{\varphi_{p-1,c}(x)\in \mathbb{Z}[x] \ : \ 5\leq p\leq c \ \text{and} \ p\mid (c-1) \ \text{for any fixed} \ c \}}{\Large{\# \{\varphi_{p-1,c}(x) \in \mathbb{Z}[x] \ : \ 5\leq p\leq c \}}}}.$
\end{center}So now, taking limit as $c\to \infty$ on both sides of the above inequality and applying Corollary \ref{6.1} and then applying a similar argument as in the Proof of Corollary \ref{5.1} to obtain an upper bound zero, we then obtain
\begin{center}
$\lim\limits_{c\to\infty}\Large{\frac{\# \{\varphi_{p-1,c}(x) \in \mathbb{Z}[x] \ : \ 5\leq p\leq c \ \text{and} \ p\mid (c-1)\}}{\Large{\# \{\varphi_{p-1,c}(x) \in \mathbb{Z}[x] \ : \ 5\leq p\leq c \}}}} = 0 = \lim\limits_{c\to\infty}\Large{\frac{\# \{\varphi_{p-1,c}(x) \in \mathbb{Z}[x] \ : \ 5\leq p\leq c \ \text{and} \ p\mid (c+1)\}}{\Large{\# \{\varphi_{p-1,c}(x) \in \mathbb{Z}[x] \ : \ 5\leq p\leq c \}}}}$ 
\end{center}where the second limit follows from also observing  $c<c+1$ and then applying an argument that's very similar to the one that has been given in the case when $c-1<c$. This then completes the whole proof, as needed.
\end{proof}

\section{On Density of Integer Monics $\varphi_{p,c}(x)$ with $N_{c}^{(n)}(p) = 0$ and $\varphi_{p-1,c}(x)$ with $M_{c}^{(n)}(p) = 0$}\label{sec8}

Recall in Corollary \ref{5.1} that a density of $0\%$ of monic integer polynomials $\varphi_{p,c}(x)$ have $p$ distinct $n$-periodic integral points modulo $p$; and so for any fixed period $n\in \mathbb{Z}_{\geq 2}$, the density of monic integer polynomials $\varphi_{p,c}^n(x)-x$ that are reducible modulo $p$ is $0\%$. So now, we also wish to determine: \say{\textit{For a prime $p\geq 3$ and for any fixed $n\in \mathbb{Z}_{\geq 2}$, what is the density of monic integer polynomials $\varphi_{p,c}(x)$ with no $n$-periodic integral points modulo $p$?}} The following corollary shows that for any fixed integer $n\geq 2$, the probability of choosing randomly a monic integer polynomial $\varphi_{p,c}(x) = x^p+c$ so that $\mathbb{Q}[x]\slash (\varphi_{p, c}^n(x)-x)$ is a degree-$p^n$ number field is equal to 1:

\begin{cor} \label{7.1}
Let $p\geq 3$ be a prime integer. Then the density of monic polynomials $\varphi_{p,c}(x) = x^p + c\in \mathbb{Z}[x]$ having the number $N_{c}^{(n)}(p) = 0$ exists and is equal to $100 \%$ as $c\to \infty$. More precisely, we have
\begin{center}
    $\lim\limits_{c\to\infty} \Large{\frac{\# \{\varphi_{p, c}(x)\in \mathbb{Z}[x] \ : \ 3\leq p\leq c \ and \ N_{c}^{(n)}(p) \ = \ 0 \}}{\Large{\# \{\varphi_{p,c}(x) \in \mathbb{Z}[x] \ : \ 3\leq p\leq c \}}}} = \ 1.$
\end{center}
\end{cor}

\begin{proof}
Since the number $N_{c}^{(n)}(p) = p$ or $0$ for any given prime integer $p\geq 3$ and since we also proved the density in Corollary \ref{5.1}, we then obtain the desired density (i.e., we obtain that the limit exists and is equal to 1). 
\end{proof}

\noindent Note that the foregoing corollary also shows that for any fixed $n\in \mathbb{Z}_{\geq 2}$, there are infinitely many polynomials $\varphi_{p,c}(x)$ over $\mathbb{Z}$ such that for $f(x) = \varphi_{p,c}^n(x)-x$, the quotient $K_{f} = \mathbb{Q}[x]\slash (f(x))$ induced by $f$ is a number field of degree $m=p^n$. Comparing the densities in Corollaries \ref{5.1} and \ref{7.1}, we may then observe that in the whole family of monic integer polynomials $\varphi_{p,c}(x) = x^p +c$, almost all such monics have no $n$-periodic integral points modulo $p$; and from which it then follows that almost all monic integer polynomials $f(x)$ are irreducible over $\mathbb{Q}$. But this may also imply that the average value of $N_{c}^{(n)}(p)$ in the whole family of polynomials $\varphi_{p,c}(x)$ is zero.

Similarly, recall in Corollary \ref{6.1} or \ref{6.2} that a density of $0\%$ of monic integer polynomials $\varphi_{p-1,c}(x)$ have $M_{c}^{(n)}(p) = 2$ or $1$, resp.; and so for any fixed period $n\in \mathbb{Z}_{\geq 2}$, the density of polynomials $\varphi_{p-1, c}^n(x)-x\in \mathbb{Z}[x]$ that are reducible modulo $p$ is $0\%$. So now, we also wish to determine: \say{\textit{For a prime $p\geq 5$ and for any fixed $n\in \mathbb{Z}_{\geq 2}$, what is the density of monic integer polynomials $\varphi_{p-1,c}(x)$ with no $n$-periodic integral points modulo $p$?}} The following corollary shows that for any fixed $n\geq 2$, the probability of choosing randomly a monic polynomial $\varphi_{p-1,c}(x)\in \mathbb{Z}[x]$ such that $\mathbb{Q}[x]\slash (\varphi_{p-1, c}^n(x)-x)$ is a number field of even degree $(p-1)^n$ is also 1:

\begin{cor} \label{7.2}
Let $p\geq 5$ be a prime integer. The density of monic polynomials $\varphi_{p-1, c}(x) = x^{p-1}+c\in \mathbb{Z}[x]$ having the number $M_{c}^{(n)}(p) = 0$ exists and is equal to $100 \%$ as $c\to \infty$. More precisely, we have 
\begin{center}
    $\lim\limits_{c\to\infty} \Large{\frac{\# \{\varphi_{p-1, c}(x)\in \mathbb{Z}[x] \ : \ 5\leq p\leq c \ and \ M_{c}^{(n)}(p) \ = \ 0 \}}{\Large{\# \{\varphi_{p-1,c}(x) \in \mathbb{Z}[x] \ : \ 5\leq p\leq c \}}}} = \ 1.$
\end{center}
\end{cor}
\begin{proof}
Recall that $M_{c}^{(n)}(p) = 1, 2$ or $0$ for any given prime $p\geq 5$ and since we also proved the densities in Corollary \ref{6.1} and \ref{6.2}, we now obtain the desired density (i.e., we get that the limit exists and is equal to 1).
\end{proof}
\noindent As before, Corollary \ref{7.2} also shows that for any fixed $n\in \mathbb{Z}_{\geq 2}$, there are infinitely many monic polynomials $\varphi_{p-1,c}(x)$ over $\mathbb{Z}$ such that for $g(x) = \varphi_{p-1,c}^n(x)-x$, the quotient $L_{g} = \mathbb{Q}[x]\slash (g(x))$ induced by $g$ is a number field of degree $r=(p-1)^n$. Again, if we compare the densities in Corollary \ref{6.1}, \ref{6.2} and \ref{7.2}, we then also note that in the whole family of monics $\varphi_{p-1,c}(x) = x^{p-1} +c\in \mathbb{Z}[x]$, almost all such monics have no $n$-periodic integral points modulo $p$; and from which it then follows that almost all monics $g(x)\in \mathbb{Z}[x]$ are irreducible over $\mathbb{Q}$. This may also imply that the average value of $M_{c}^{(n)}(p)$ in the whole family of monics $\varphi_{p-1,c}(x)$ is also zero.

Recall more generally that any number field $K$ is always naturally equipped with a ring $\mathcal{O}_{K}$ of integers in $K$; and which is classically known to describe the arithmetic of $K$, but usually very difficult to compute in practice. So now, it then follows that $K_{f}$ has a ring of integers $\mathcal{O}_{K_{f}}$ and moreover applying (as in \cite{BK3, BK1}) a theorem due to Bhargava-Shankar-Wang [\cite{sch1}, Theorem 1.2], we then also have the following corollary showing that the probability of choosing randomly an irreducible integer polynomial $f$ arising from a polynomial discrete dynamical system in Section \ref{sec2}, such that the quotient $\mathbb{Z}[x]\slash (f(x))$ is the ring of integers of $K_{f}$, is $\approx 60.7927\%$:
\begin{cor} \label{8.2}
Assume Corollary \ref{7.1}. When monic integer polynomials $f(x)$ are ordered by height $H(f)$ as defined in \textnormal{\cite{sch1}}, the density of such polynomials $f(x)$ such that $\mathbb{Z}[x]\slash (f(x))$ is the ring of integers of $K_{f}$ is $\zeta(2)^{-1}$. 
\end{cor}

\begin{proof}
From Corollary \ref{7.1}, there are infinitely many irreducible monic polynomials $f(x)$ over $\mathbb{Z}$ (and hence over $\mathbb{Q}$) such that $K_{f} = \mathbb{Q}[x]\slash (f(x))$ is an algebraic number field of deg$(f) = p^n$; and moreover associated to $K_{f}$ is the ring of integers $\mathcal{O}_{K_{f}}$. This also means that the set of irreducible monic polynomials $f(x)=\varphi_{p,c}^n(x)-x\in \mathbb{Z}[x]$ such that $K_{f}$ is a number field of degree $m=p^n$, is not empty. So now, applying [\cite{sch1}, Theorem 1.2] to the underlying family of monic polynomials $f(x)\in \mathbb{Z}[x]$ ordered by height $H(f)$ as defined in \cite{sch1} such that $\mathcal{O}_{K_{f}} = \mathbb{Z}[x]\slash (f(x))$, it then follows that the density of such polynomials $f(x)\in \mathbb{Z}[x]$ is equal to $\zeta(2)^{-1} \approx 60.7927\%$, as desired.
\end{proof}

Similarly, we note that every number field $L_{g}$ induced by a polynomial $g$, is also naturally equipped with the ring of integers $\mathcal{O}_{L_{g}}$, and which may also be difficult to compute in practice. So now as before, we note that by again taking great advantage of [\cite{sch1}, Theorem 1.2], we then also obtain the following corollary which shows that the probability of choosing randomly an irreducible monic  integer polynomial $g$ arising from a polynomial discrete dynamical system in Section \ref{sec3}, such that $\mathbb{Z}[x]\slash (g(x))$ is the ring of integers of $L_{g}$ is also $\approx 60.7927\%$:

\begin{cor}
Assume Corollary \ref{7.2}. When monic integer polynomials $g(x)$ are ordered by height $H(g)$ as defined in \textnormal{\cite{sch1}}, the density of such polynomials $g(x)$ such that $\mathbb{Z}[x]\slash (g(x))$ is the ring of integers of $L_{g}$ is $\zeta(2)^{-1}$. 
\end{cor}

\begin{proof}
By applying a similar argument as in the Proof of Cor. \ref{8.2}, we then obtain the density, as desired.  
\end{proof}

\section{On Number of Monic Polynomials $f\in \mathbb{Z}[x]$ \& $g\in \mathbb{Z}[x]$ with Primitive Galois groups}\label{sec9}

Recall from Corollary \ref{7.1} that there is an infinite family of irreducible monic integer polynomials $f(x) = \varphi_{p,c}^n(x)-x$ such that the quotient $K_{f}=\mathbb{Q}[x]\slash (f(x))$ induced by $f$ is a number field of degree $m=p^n$. Moreover, to each such irreducible monic integer polynomial $f\in \mathbb{Q}[x]$, let $G_{f}$ be the Galois group of $f$ over $\mathbb{Q}$. 

So now, inspired by recent impressive work of Bhargava \cite{gav} on van der Waerden's Conjecture, we then also wish to determine the number of irreducible monic polynomials  $f\in \mathbb{Z}[x]$ arising from a polynomial discrete dynamical system in Section \ref{sec2}, of bounded height and such that $G_{f}$ is a primitive Galois group not equal to the full symmetric group $S_{m}$. To that end, we (assuming Corollary \ref{7.1}) wish to first adhere to the setup and definition of coefficient height $h(f)$ in \cite{gav}). That is, for any fixed $m=p^n$, we let $E_{m}(H)$ be the number of monic integer polynomials $f(x) = \varphi_{p,c}^n(x)-x$ of degree $m$ with $h(f)\leq H$ and such that $G_{f}\neq S_{m}$. Now by taking great advantage of a theorem of Bhargava [\cite{gav}, Theorem 1], we then obtain the following corollary on $E_{m}(H)$:

\begin{cor}\label{9.1}
Assume Corollary \ref{7.1}, and let $E_{m}(H)$ be defined as before. Then we have $E_{m}(H)=O(H^{m-1})$.
\end{cor}

\begin{proof}
From Corollary \ref{7.1}, there are infinitely many irreducible monic polynomials $f(x)=\varphi_{p,c}^n(x)-x\in \mathbb{Z}[x]$ of degree $m=p^n$. Now for every monic polynomial $f\in \mathbb{Z}[x]\subset \mathbb{Q}[x]$ of fixed degree $m=p^n$, let $G_{f}$ be the Galois group of $f$ over $\mathbb{Q}$. But now applying [\cite{gav}, Theorem 1] on the set of monic polynomials $f\in \mathbb{Z}[x]$ of degree $m$ with $h(f)\leq H$ and such that $G_{f}$ is primitive and $G_{f}\neq S_{m}$, we then immediately obtain the count, as needed.
\end{proof}

Similarly, we may also recall from Corollary \ref{7.2} that there is an infinite family of irreducible monic integer polynomials $g(x) = \varphi_{p-1,c}^n(x)-x$ such that the quotient $L_{g}=\mathbb{Q}[x]\slash (g(x))$ induced by $g$ is a number field of degree $r=(p-1)^n$. And moreover, to each such irreducible integer polynomial $g\in \mathbb{Q}[x]$, let $G_{g}$ be the Galois group of $g$ over $\mathbb{Q}$. So now, inspired again work of Bhargava \cite{gav} on van der Waerden's Conjecture, we then also wish to determine the number of irreducible monic polynomials  $g\in \mathbb{Z}[x]$ arising from a polynomial discrete dynamical system in Section \ref{sec3}, of bounded height and such that $G_{g}$ is a primitive Galois group not equal to the full symmetric group $S_{r}$. To that end, we (assuming Corollary \ref{7.2}) as before first import the setup and definition of coefficient height $h(g)$ in \cite{gav}). That is, for any fixed $r=(p-1)^n$, we let $E_{r}(H)$ be the number of monic integer degree-$r$ polynomials $g(x) = \varphi_{p-1,c}^n(x)-x$ with $h(g)\leq H$ and such that $G_{g}\neq S_{r}$. So now, by again taking great advantage of Bhargava's theorem [\cite{gav}, Thm. 1], we then also obtain the following corollary:

\begin{cor}
Assume Corollary \ref{7.2}, and let $E_{r}(H)$ be defined as before. Then we have $E_{r}(H)=O(H^{r-1})$.
\end{cor}

\begin{proof}
By applying a similar argument as in the Proof of Corollary \ref{9.1}, we then obtain the conclusion, as needed. That is, from Corollary \ref{7.2}, there are infinitely many irreducible monic polynomials $g(x)=\varphi_{p-1,c}^n(x)-x\in \mathbb{Z}[x]$ of degree $r=(p-1)^n$. So now, for every monic $g\in \mathbb{Z}[x]\subset \mathbb{Q}[x]$ of fixed degree $r=(p-1)^n$, let $G_{g}$ be the Galois group of $g$ over $\mathbb{Q}$. But now applying [\cite{gav}, Theorem 1] on the set of monic degree-$r$ polynomials $g\in \mathbb{Z}[x]$ with $h(g)\leq H$ and such that $G_{g}$ is primitive and $G_{g}\neq S_{r}$, we then immediately obtain the count, as needed. 
\end{proof}

\section{Fields $K_{f}$ and $L_{g}$ with Bounded Absolute Discriminant \&  Prescribed Galois group}\label{sec10}

As in Section \ref{sec9}, recall from Corollary \ref{7.1} that there is an infinite family of irreducible monic integer polynomials $f(x) = \varphi_{p,c}^n(x)-x$ such that the field $K_{f}$ induced by $f$ is an algebraic number field of  degree $m=p^n$. Similarly, recall from Corollary \ref{7.2} that one can always find an infinite family of irreducible monic integer polynomials $g(x) = \varphi_{p-1,c}^n(x)-x$ such that the field extension $L_{g}$ over $\mathbb{Q}$ induced by $g$ is an algebraic number field of degree $r=(p-1)^n$. Moreover, from standard theory of number fields, we may associate to $K_{f}$ (resp., $L_{g}$) an integer Disc$(K_{f})$ (resp., Disc$(L_{g})$) called the discriminant. So now inspired (as in \cite{BK3}) by number field-counting advances in arithmetic statistics, we also wish to count the number of fields $K_{f}$ and $L_{g}$ induced by irreducible monic polynomials $f$ and $g$ arising from polynomial discrete dynamical systems in Section \ref{sec2} and \ref{sec3}. To do so, we (as in \cite{BK3}) define and then also determine the asymptotic behavior of the following counting functions  
\begin{equation}\label{N_{m}}
N_{m}(X) := \# \Bigl\{K_{f}\slash \mathbb{Q} : [K_{f} : \mathbb{Q}] = m \textnormal{ and } |\text{Disc}(K_{f})|\leq X \Bigr\}
\end{equation} 
\begin{equation}\label{M_{r}}
M_{r}(X) := \# \Bigl\{L_{g}\slash \mathbb{Q} : [L_{g} : \mathbb{Q}] = r \textnormal{ and} \ |\text{Disc}(L_{g})|\leq X \Bigr\}
\end{equation} as a positive real number $X\to \infty$. So now, motivated (as in \cite{BK3}) by work of Lemke Oliver-Thorne \cite{lem} and then applying here the first part of their [\cite{lem}, Theorem 1.2] on the function $N_{m}(X)$, we then obtain the following:

\begin{cor} \label{8.5} Assume Corollary \ref{7.1}, and let $N_{m}(X)$ be the number defined as in \textnormal{(\ref{N_{m}})}. Then we have 
\begin{equation}\label{N_{m}(x)} 
N_{m}(X)\ll_{m}X^{2d - \frac{d(d-1)(d+4)}{6m}}\ll X^{\frac{8\sqrt{m}}{3}}, \text{where d is the least integer for which } \binom{d+2}{2}\geq 2m + 1.
\end{equation}
\end{cor}

\begin{proof}
To see the inequality \textnormal{(\ref{N_{m}(x)})}, we first recall from Cor. \ref{7.1} the existence of infinitely many monic integer polynomials $f(x)\in \mathbb{Q}[x]$ such that $K_{f}\slash \mathbb{Q}$ is a number field of degree $m=p^n$. Hence, the set of fields $K_{f}\slash \mathbb{Q}$ is not empty. So now, we may apply [\cite{lem}, Thm 1.2 (1)] on $N_{m}(X)$ to then obtain the upper bound, as required.
\end{proof}

Motivated again by that same work of Lemke Oliver-Thorne \cite{lem}, we again take great advantage of the first part of [\cite{lem}, Theorem 1.2] by applying it on $M_{r}(X)$. In doing so, we then obtain the following corollary:

\begin{cor}Assume Corollary \ref{7.2}, and let $M_{r}(X)$ be the number defined as in \textnormal{(\ref{M_{r}})}. Then we have 
\begin{equation}\label{M_{r}(x)}
M_{r}(X)\ll_{r}X^{2d - \frac{d(d-1)(d+4)}{6r}}\ll X^{\frac{8\sqrt{r}}{3}}, \text{where d is the least integer for which } \binom{d+2}{2}\geq 2r + 1.
\end{equation}
\end{cor}

\begin{proof}
By applying a similar argument as in the Proof of Cor. \ref{8.5}, we then obtain inequality \textnormal{(\ref{M_{r}(x)})}, as needed.
\end{proof}

We recall more generally that an algebraic number field $K$ is  \say{\textit{monogenic}} if there exists an algebraic number $\alpha \in K$ such that the ring of integers $\mathcal{O}_{K}$ is the subring $\mathbb{Z}[\alpha]$ generated by $\alpha$ over $\mathbb{Z}$, i.e., $\mathcal{O}_{K}= \mathbb{Z}[\alpha]$. So now, inspired (as in \cite{BK3}), we also wish to count the number of number fields $K_{f}$ induced by irreducible monic polynomials $f$ arising from a polynomial discrete dynamical system in Section \ref{sec2}, that are monogenic with $|\Delta(K_{f})| < X$ and have associated Galois group Gal$(K_{f}\slash \mathbb{Q})$ equal to symmetric group $S_{p^n}$. To this end, we (as in \cite{BK3}) take great advantage of a result due to Bhargava-Shankar-Wang [\cite{sch1}, Corollary 1.3] and then  obtain:

\begin{cor}\label{8.3}
Assume Corollary \ref{7.1}. The number of isomorphism classes of algebraic number fields $K_{f}$ of odd degree $m=p^n$ and with $|\Delta(K_{f})| < X$ that are monogenic and have associated Galois group $S_{m}$ is $\gg X^{\frac{1}{2} + \frac{1}{m}}$.
\end{cor}

\begin{proof}
By Cor. \ref{7.1}, it then also follows that there are infinitely many monic polynomials $f(x)$ over $\mathbb{Z}$ (and over $\mathbb{Q}$) such that the quotient $K_{f}$ induced by $f$ is an algebraic number field of odd degree $m=p^n$. This also means that the set of algebraic number fields $K_{f}\slash \mathbb{Q}$ of odd degree $m$ is not empty. So now, applying [\cite{sch1}, Cor. 1.3] to the underlying number fields $K_{f}$ with $|\Delta(K_{f})| < X$ that are monogenic and have associated Galois group $S_{m}$, we then obtain that the number of isomorphism classes of such number fields $K_{f}$ is $\gg X^{\frac{1}{2} + \frac{1}{m}}$, as required.
\end{proof}

Similarly, we again take great advantage of that same result due to Bhargava-Shankar-Wang [\cite{sch1}, Corollary 1.3] to then also immediately count in the following corollary the number of number fields $L_{g}$ induced by irreducible monic polynomials $g$ arising from a polynomial discrete dynamical system in Section \ref{sec3}, that are monogenic with $|\Delta(L_{g})| < X$ and have associated Galois group Gal$(L_{g}\slash \mathbb{Q})$ equal to symmetric group $S_{(p-1)^n}$: 

\begin{cor}
Assume Corollary \ref{7.2}. The number of isomorphism classes of algebraic number fields $L_{g}$ of even degree $r=(p-1)^n$ and $|\Delta(L_{g})| < X$ that are monogenic and have associated Galois group $S_{r}$ is $\gg X^{\frac{1}{2} + \frac{1}{r}}$.
\end{cor}

\begin{proof}
Applying a similar argument as in the Proof of Corollary \ref{8.3}, we then obtain the count, as required.
\end{proof}

\section{On Number of Algebraic Number fields $K_{f}$ and $L_{g}$ with Prescribed Class Number}\label{sec11}

Recall that for any number field $K$ with ring of integers $\mathcal{O}_{K}$, we have a finite abelian group called \say{\textit{ideal class group}} $\textnormal{Cl}(\mathcal{O}_{K})$ (also denoted as $\textnormal{Cl}(K)$), which is classically known to provide a way of measuring how far $\mathcal{O}_{K}$ is from being a unique factorization domain. Now even though the order (also called the \say{\textit{class number}} of $K$ (denoted as $h_{K}$)) of $\textnormal{Cl}(\mathcal{O}_{K})$ is finite, it is well known in algebraic and analytic number theory and even more so in arithmetic statistics, that computing $\textnormal{Cl}(\mathcal{O}_{K})$ in practice let alone determine precisely $h_{K}$, is a hard problem. 

So now, recall from Corollary \ref{7.1} that there is an infinite family of irreducible monic integer polynomials $f(x) = \varphi_{p,c}^n(x)-x$ such that $K_{f}=\mathbb{Q}[x]\slash (f(x))$ is a number field of odd degree $p^n$; and moreover to each $K_{f}$, we also have $\textnormal{Cl}(K_{f})$ with finite $h_{K_{f}}$. Now inspired by remarkable work of Ho-Shankar-Varma \cite{ho} on odd degree number fields with odd class number, we then wish to count the number of number fields $K_{f}$ induced by irreducible polynomials $f$ arising from a polynomial discrete dynamical system in Section \ref{sec2}, with associated Galois group $S_{p^{n}}$ and with prescribed $h_{K_{f}}$. To do so, we take great advantage of [\cite{ho}, Theorem 4] and then obtain the following corollary on the existence of infinitely many $S_{p^{n}}$-number fields $K_{f}$ with odd class number:

\begin{cor}
Assume Corollary \ref{7.1}, and let $m=p^n$ be any fixed odd integer. Then there exist infinitely many $S_{m}$-algebraic number fields $K_{f}$ of odd degree $m$  having odd class number.  More precisely, we have 
\begin{center}
$\#\Bigl\{ K_{f} : |\Delta(K_{f})| < X \textnormal{ and } 2\nmid \textnormal{Cl}(K_{f})|\Bigr\}\gg X^{\frac{m + 1}{2m -2}}$,
\end{center} where the implied constants depend on degree $m$ and on an arbitrary finite set $S$ of primes given as in \textnormal{\cite{ho}}.
\end{cor}

\begin{proof}
From Cor. \ref{7.1}, it follows that the family of number fields $K_{f}$ of degree $m = p^n$ is not empty. Now since $m$ is an odd integer, we then see that the claim follows from [\cite{ho}, Thm. 4(a)] by setting $K_{f}=K$ as needed.
\end{proof}

As before, we may also recall from Corollary \ref{7.2} the existence of an infinite family of irreducible monic integer polynomials $g(x) = \varphi_{p-1,c}^n(x)-x$ such that $L_{g} = \mathbb{Q}[x]\slash (g(x))$ induced by $g$ is an algebraic number field of even degree $(p-1)^n$; and moreover to each $L_{g}$, we also have $\textnormal{Cl}(L_{g})$ with finite $h_{L_{g}}$. So now, by taking again great advantage of work on class groups of number fields in arithmetic statistics and in particular the nice work of Siad \cite{Sia} on $S_{n}$-number fields $K$ of any even degree $n\geq 4$ and signature $(r_{1}, r_{2})$ where $r_{1}$ are the real embeddings of $K$  and $r_{2}$ are the pairs of conjugate complex embeddings of $K$, we then also obtain the following corollary on the number of fields $L_{g}\slash \mathbb{Q}$ induced by irreducible monic polynomials $g$ arising from a polynomial discrete dynamical system in Section \ref{sec3}, with associated Galois group $S_{(p-1)^n}$ and also having odd class number: 

\begin{cor}
Assume Corollary \ref{7.2}, and let $r=(p-1)^n$ be an even integer. There are infinitely many degree-$r$ monogenic number fields $L_{g}$ of any signature and associated Galois group $S_{r}$ having odd class number. 
\end{cor}
\begin{proof}
To see this, we note that by Cor. \ref{7.2}, it follows that the family of number fields $L_{g}$ of degree $r = (p-1)^n$ is not empty. Now since $r$ is even, we then note that the claim follows from [\cite{Sia}, Cor. 10], as indeed needed.
\end{proof}

\section{On Equidistribution of Families of Artin $L$-Functions induced by Fields $K_{f}$ and $L_{g}$}\label{sec12}

Recall that for any degree-$n$ every number field $K$ with ring of integers $\mathcal{O}_{K}$, we have a Dedekind zeta function $\zeta_{K}$ associated with $K$; and which for every complex $s\in \mathbb{C}$ with $\mathfrak{R}(s)>1$, this zeta function $\zeta_{K}$ is defined by  
\begin{equation}
    \zeta_{K}(s) = \prod_{\mathfrak{p}\subset \mathcal{O}_{K}}\frac{1}{1-|\mathcal{O}_{K}\slash \mathfrak{p}|^{-s}}   
\end{equation}where the product is taken over all the nonzero prime ideals $\mathfrak{p}$, and $|\mathcal{O}_{K}\slash \mathfrak{p}|$ is the absolute norm of $\mathfrak{p}$. As a generalization of the Riemann zeta function $\zeta_{\mathbb{Q}}(s)$ (whose vanishing on the line  $\mathfrak{R}(s) = \frac{1}{2}$ is intimately related to the distribution of primes $p \in \mathbb{Z}$ (as a consequence of the Riemann Hypothesis)), it is a classical theme in number theory to understand the vanishing of $\zeta_{K}(s)$ especially on the line $\mathfrak{R}(s) = \frac{1}{2}$, since it is also expected that the vanishing of $\zeta_{K}(s)$ on the line $\mathfrak{R}(s) = \frac{1}{2}$ reveals precise information about the distribution of primes $\mathfrak{p}$ in $K$ (as a consequence of the number field version of the Riemann Hypothesis). We note that from [\cite{Nico}, Page 10] the zeta function $\zeta_{K}(s)$ factors as $\zeta_{K}(s)=\zeta_{\mathbb{Q}}(s)L(s, \rho_{K}$), where $L(s, \rho_{K})$ is the Artin $L$-function corresponding to an Artin representation $\rho_{K}: \text{Gal}(\mathbb{Q})\to \text{Gal}(M\slash \mathbb{Q}) \hookrightarrow S_{n}\to \text{GL}_{n-1}(\mathbb{C})$, and $M$ is the normal closure of $K$.

So now, for every degree-$m$ number field $K_{f}$ obtained from a polynomial discrete dynamical system in Section \ref{sec2} and ascertained by Corollary \ref{7.1}, we have a Dedekind zeta function $\zeta_{K_{f}}$ corresponding to $K_{f}$. Moreover, we also know from the remarkable work of Shankar-S\"{o}dergren-Templier [\cite{Nico}, Page 2] that the zeta function $\zeta_{K_{f}}(s)=\zeta(s)L(s, \rho_{K_{f}}$), where $\zeta(s)$ is the Riemann zeta function, $L(s, \rho_{K_{f}})$ is the Artin $L$-function, $\rho_{K_{f}}: \text{Gal}(M_{f}\slash \mathbb{Q}) \hookrightarrow S_{m}\to \text{GL}_{m-1}(\mathbb{C})$ is an Artin representation, and where $M_{f}$ is the normal closure of $K_{f}$. 

Now inspired by remarkable work of Shankar-S\"{o}dergren-Templier \cite{Nico} on equidistribution of Artin $L$-functions arising from number fields induced by irreducible monic integer polynomials, we in the exact same spirit as in \cite{Nico} also wish to study the distribution of Artin $L$-functions $L(s, \rho_{K_{f}})$ arising from number fields $K_{f}$ induced by irreducible monic polynomials $f$ obtained from a polynomial discrete dynamical system in Section \ref{sec2}. To do so, we (assuming Corollary \ref{7.1}) wish to first adhere to the setup and notation in \cite{Nico}. That is, let $V(\mathbb{Z})^{\text{irr}}$ be the space consisting of irreducible monic integer polynomials  $f(x)=\varphi_{p,c}^n(x)-x$ of fixed degree $m=p^n$,  and let $V(\mathbb{Z})^{\text{max}}\subset V(\mathbb{Z})^{\text{irr}}$ be a subset consisting of irreducible monic integer polynomials $f$ such that $R_{f}=\mathbb{Z}[x]\slash (f(x))$ is a maximal order in $K_{f}=\mathbb{Q}[x]\slash (f(x))$. Following \cite{Nico}, it also follows here that the additive group $G_{a}(\mathbb{Z})=\mathbb{Z}$ necessarily acts naturally on our space $V(\mathbb{Z})^{\text{irr}}$ via translation, namely, $(b \cdot f)(x):= f(x+b)$ for every element $b\in \mathbb{Z}$ and for every $f\in V(\mathbb{Z})^{\text{irr}}$; and moreover, this action of $G_{a}(\mathbb{Z})=\mathbb{Z}$ by translation also necessarily preserves each of the sets $V(\mathbb{Z})^{\text{irr}}$ and $V(\mathbb{Z})^{\text{max}}$. Now let $\mathfrak{F}_{1}$ be a family consisting of the $\mathbb{Z}$-orbits on $V(\mathbb{Z})^{\text{max}}$. It then follows (from \cite{Nico}) that the family $\mathfrak{F}_{1}$ necessarily parametrizes degree-$m$ monogenized number fields $(K_{f}, \alpha)$ over $\mathbb{Q}$ up to isomorphism. Note that (by [\cite{Nico}, Subsection 2.3]) this same family $\mathfrak{F}_{1}$ parametrizing  degree-$n$ monogenized fields $(K_{f}, \alpha)$ is also treated to be the same family of corresponding $L$-functions $L(s, \rho_{K_{f}})$.

So now, by taking great advantage of a nice theorem of Shankar-S\"{o}dergren-Templier[\cite{Nico}, Theorem 1.1], we also then obtain the following corollary on the family $\mathfrak{F}_{1}$ parametrizing degree-$m$ monogenized fields $(K_{f}, \alpha)$:

\begin{cor}\label{11.1}
Assume Corollary \ref{7.1}, and let $\mathfrak{F}_{1}$ be as before. Then $\mathfrak{F}_{1}$ parametrizing monogenized degree-$m$ fields ordered by height $h(f)$ as defined in \textnormal{\cite{Nico}} satisfies Sato-Tate equidistribution in the sense of \textnormal{[\cite{Sar}, Conj.1]}. 
\end{cor}

\begin{proof}
Since we know from Corollary \ref{7.1} that there are infinitely many irreducible monic integer polynomials $f$ such that $K_{f}$ is a number field of degree $m=p^n$, then this also means that the family of degree-$m$ number fields $K_{f}\slash \mathbb{Q}$ is not empty. Now letting $\alpha$ be the image of $x$ in $R_{f}=\mathbb{Z}[x]\slash (f(x))$ and so (by \cite{Nico}) the pair $(K_{f}, \alpha)$ is a degree-$m$ monogenized field, it then follows that the family of monogenized degree-$m$ fields $(K_{f}, \alpha)$ is not empty; which also means that the family $\mathfrak{F}_{1}$ parametrizing  degree-$m$ monogenized fields $(K_{f}, \alpha)$ is not empty. But now applying [\cite{Nico}, Thm. 1.1] to the underlying family $\mathfrak{F}_{1}$ ordered by height $h(f)$ as defined in [\cite{Nico}, Page 3], it then follows that $\mathfrak{F}_{1}$ satisfies Sato-Tate equidistribution in the sense of \textnormal{[\cite{Sar}, Conjecture 1]}, as needed.
\end{proof}

Similarly, for every degree-$r$ field $L_{g}$ obtained from a polynomial discrete dynamical system in Section \ref{sec3} and ascertained by Corollary \ref{7.2}, we also have a Dedekind zeta function $\zeta_{L_{g}}$ corresponding to $L_{g}$. Moreover, we again also know from \cite{Nico} that the zeta function $\zeta_{L_{g}}(s)=\zeta(s)L(s, \rho_{L_{g}}$), where $L(s, \rho_{L_{g}})$ is the Artin $L$-function,  $\rho_{\mathbb{Q}_{g}}: \text{Gal}(M_{g}\slash \mathbb{Q}) \hookrightarrow S_{r}\to \text{GL}_{r-1}(\mathbb{C})$ is an Artin representation, and where $M_{g}$ being the normal closure of $L_{g}$. 

So now, in again the same spirit as in \cite{Nico}, we also wish to study the distribution of Artin $L$-functions $L(s, \rho_{L_{g}})$ arising from fields $L_{g}$ induced by irreducible polynomials $g$ obtained from a polynomial discrete dynamical system in Section \ref{sec3}. To do so, we (also assuming Corollary \ref{7.2}) adhere again to the setup and notation in \cite{Nico}. That is, we let $W(\mathbb{Z})^{\text{irr}}$ be the space consisting of irreducible monic integer polynomials  $g(x)=\varphi_{p-1,c}^n(x)-x$ of fixed degree $r=(p-1)^n$,  and let $W(\mathbb{Z})^{\text{max}}\subset W(\mathbb{Z})^{\text{irr}}$ be a subset consisting of irreducible polynomials $g$ such that $R_{g}=\mathbb{Z}[x]\slash (g(x))$ is a maximal order in $L_{g}=\mathbb{Q}[x]\slash (g(x))$. Following again \cite{Nico}, it also follows here that $G_{a}(\mathbb{Z})=\mathbb{Z}$ necessarily acts naturally on $W(\mathbb{Z})^{\text{irr}}$ via translation, namely, $(b \cdot g)(x):= g(x+b)$ for every $b\in \mathbb{Z}$ and for every $g\in W(\mathbb{Z})^{\text{irr}}$; and moreover, this action of $G_{a}(\mathbb{Z})=\mathbb{Z}$ by translation also necessarily preserves each of $W(\mathbb{Z})^{\text{irr}}$ and $W(\mathbb{Z})^{\text{max}}$. Now let $\mathfrak{F}_{2}$ be a family consisting of the $\mathbb{Z}$-orbits on $W(\mathbb{Z})^{\text{max}}$. It then follows (from \cite{Nico}) that the family $\mathfrak{F}_{2}$ necessarily parametrizes degree-$r$ monogenized fields $(L_{g}, \beta)$ up to isomorphism. As before, we also note that (from [\cite{Nico}, Subsect.2.3]) this same family $\mathfrak{F}_{2}$ parametrizing  degree-$r$ monogenized fields $(L_{g}, \beta)$ is also the family of associated $L$-functions $L(s, \rho_{L_{g}})$. By again, taking great advantage of [\cite{Nico}, Theorem 1.1], we then obtain the following corollary on the family $\mathfrak{F}_{2}$:

\begin{cor}
Assume Corollary \ref{7.2}, and let $\mathfrak{F}_{2}$ be as before. Then $\mathfrak{F}_{2}$ parametrizing monogenized degree-$r$ fields ordered by height $h(g)$ as defined in \textnormal{\cite{Nico}} satisfies Sato-Tate equidistribution in the sense of \textnormal{[\cite{Sar}, Conj.1]}. 
\end{cor}

\begin{proof}
By applying a similar argument as in the Proof of Corollary \ref{11.1}, it then also follows immediately that the family $\mathfrak{F}_{2}$ satisfies  Sato-Tate equidistribution in the sense of \textnormal{[\cite{Sar}, Conjecture 1]}, as also indeed needed.
\end{proof}

\section*{\textbf{Acknowledgments}}
I'm truly very grateful and deeply indebted to my long-time great advisors, Dr. Ilia Binder and Dr. Arul Shankar, and along with Dr. Jacob Tsimerman for everything. I'm very grateful to Prof. Shankar for coordinating an enlightening nine weeks number theory learning seminar on Bhargava's breakthrough work \cite{gav} in the Fall 2022. As a graduate research student, this work and my studies were hugely and wholeheartedly funded by Dr. Binder and Dr. Shankar. Any opinions expressed in this article belong solely to me, the author, Brian Kintu; and should never be taken as a reflection of the views of anyone that's been  acknowledged by the author. 

\bibliography{References}
\bibliographystyle{plain}

\noindent Dept. of Math. and Comp. Sciences (MCS), University of Toronto, Mississauga, Canada \newline
\textit{E-mail address:} \textbf{brian.kintu@mail.utoronto.ca}\newline 
\date{\small{\textit{January 3, 2026}}}

\end{document}